\documentclass[11p]{elsarticle}

\usepackage{amsmath}
\usepackage{amsfonts}
\usepackage{verbatim}
\usepackage{amsbsy}
\usepackage{multirow}
\usepackage{color,graphicx}
\usepackage{amsthm}
\usepackage{makecell}
\usepackage{tikz}
\usetikzlibrary{positioning,arrows.meta}
\usetikzlibrary{calc}
\usepackage{relsize}
\usepackage{scalefnt}
\usepackage{tabularx}

\usepackage{amssymb, amscd} 
\usepackage{algorithm, algpseudocode}
\usepackage{listings}
\usepackage{booktabs}
\usepackage{xcolor}
\usepackage{hyperref}
\usepackage{graphicx}
\usepackage[font=small]{caption} 
\usepackage{subcaption}
\usepackage{bm} 

\usepackage{bm}
\usepackage{bbm}

\numberwithin{equation}{section}
\newtheorem{proposition}{Proposition}
\newtheorem{remark}{Remark}
\newtheorem{lemma}{Lemma}

\newcommand{\vect}[1]{\ensuremath{\bm{#1}}}

\definecolor{codegreen}{rgb}{0,0.6,0}
\definecolor{codegray}{rgb}{0.5,0.5,0.5}
\definecolor{codepurple}{rgb}{0.58,0,0.82}
\definecolor{backcolour}{rgb}{0.95,0.95,0.92}

\lstdefinestyle{mystyle}{
    backgroundcolor=\color{backcolour},
    commentstyle=\color{codegreen},
    keywordstyle=\color{magenta},
    numberstyle=\tiny\color{codegray},
    stringstyle=\color{codepurple},
    basicstyle=\ttfamily\footnotesize,
    breakatwhitespace=false,
    breaklines=true,                 
    captionpos=b,           
    keepspaces=true,        
    numbers=left,       
    numbersep=5pt,     
    showspaces=false,    
    showstringspaces=false,
    showtabs=false,    
    tabsize=2
}
\journal{Computers and Mathematics with Applications}

\begin{document}

\begin{frontmatter}

\title{{A Conjugate Gradient Formulation of the EnKF Algorithm}}

\author[fsu]{Sanghyun Lee}
\author[penn]{Zhengqi Liu}
\author[fsu]{Jonathan Valyou}
\author[penn]{Ludmil Zikatanov}

\affiliation[fsu]{organization={Department of Mathematics, Florida State University},
            city={Tallahassee}, 
            state={FL},
            country={USA}}

\affiliation[penn]{organization={Department of Mathematics, Pennsylvania State University},
            city={University Park},
            state={PA},
            country={USA}}

\begin{abstract}
Ensemble Kalman Filter (EnKF) based data assimilation algorithms synthesize predictive numerical forecast models with accumulated data as time evolves and account for model uncertainty and noisy measurements.  The computational cost of these algorithms can be expensive, in particular for highly dimensional dynamical systems.  Often, EnKF based algorithms have traded accuracy for reduced computational cost.  In this paper, we present a novel parallelizable Conjugate Gradient-based Ensemble Kalman Filter (CGD-EnKF) algorithm that maintains comparable computational cost to efficient algorithms while realizing better state estimation accuracy in select cases.  Here, we established the new approach by reformulating a matrix inverse calculation with a classical Conjugate Gradient (CGD) method.  In addition, we discuss the upper error bound under CGD, error convergence to the classical EnKF result, and the computational complexity of the algorithm.  We also showcase the CGD-EnKF-Reduced algorithm that is shown to be further computationally efficient for highly dimensional dynamical systems under small ensemble formulation.  Numerical examples demonstrate the performance of our proposed algorithms and analytical properties, highlighting their comparability and advantages with respect to some benchmark EnKF algorithms.

\end{abstract}


\end{frontmatter}

\section{Introduction}
\label{sec:intro}

Efficiently estimating the state of a dynamical system given the presence of uncertainty is a primary challenge that spans many fields, including climate modeling, robotics, and subsurface porous media \cite{Bannister2020, Smith1987, Hu2025, Arbogast1990}.  Data assimilation techniques aim to address this challenge by coupling numerical models with uncertainty and noisy observational data to improve the accuracy of state estimates.  One of the most well-known data assimilation algorithms is the Kalman Filter (KF) \cite{Kalman1960}.  Derived from linear Gaussian assumptions, KF produces optimal state estimates based upon evolving the state mean and covariance in time.

Several algorithm extensions to KF have been developed to address the realistic presence of nonlinearity and non-Gaussian uncertainty.  
The Extended Kalman Filter (EKF) extended KF to nonlinear Gaussian systems through a first-order linearization \cite{Schmidt1962}, and the Ensemble Kalman Filter (EnKF) improved estimations for nonlinear Gaussian systems through ensemble-based Monte Carlo forecasting techniques \cite{Evensen1994, Anderson2001, Houtekamer2005}. The EnKF algorithm is also able to produce accurate state estimations given slightly non-Gaussian uncertainty systems \cite{Anderson2010}.

While these extensions allowed for greater applicability of the Kalman-based data assimilation framework, they also introduced significantly higher computational complexity, in particular for state estimation of high-dimensional systems \cite{Tippett2003, Sakov2008}.  A primary source of the computational cost is the formation of the innovation matrix when computing the Kalman Gain, a crucial matrix that determines how heavily observations are weighted compared to model predictions when computing the state estimation.  To address this, the current state-of-the-art EnKF in terms of computational efficiency without utilizing localization is the Serial EnKF (sEnKF) \cite{Houtekamer2001}.  The sEnKF algorithm assimilates a single observation state sequentially within a given observation rather than assimilating the entire observation vector simultaneously during a time step.  This adjusted algorithm reduces both computational cost and memory requirements by reducing the matrix inversion computation to a rank-1 scalar update that utilizes Kalman Gain vectors instead of the full Kalman Gain matrix. sEnKF is parallelizable, allowing for further time cost reduction in the presence of multiple processors.

{A second standard state-of-the-art algorithm that improves computational efficiency is the Local Ensemble Transform Kalman Filter (LETKF) \cite{Hunt2007, Szunyogh2008}.  LETKF performs the EnKF algorithm on local regions of the domain allowing for local EnKF analysis updates considering lower dimensional space.  This method is also parallelizable furthering its efficiency.  The method is intended to remove spurious correlations, yet often can bias the covariance estimation through suppressing real correlations as well potentially resulting in inaccurate state estimations.  While this algorithm is worth noting for completeness of the literature, it parallel processes the data across the localized domains, which is a distinct way from sEnKF and is distinct from our tested algorithms.  Therefore, we will aim to test our algorithms against the more comparable sEnKF in terms of parallelization.}

Further work to reduce the computational cost of Kalman-based algorithms while maintaining state estimation accuracy has been explored \cite{Bishop2001, evensen2003ensemble, Mandel2006, Nino-Ruiz2015}.  One notable early approach used projection methods to obtain low-rank approximations of the covariance matrices \cite{Dee1991, Cane1996}. 
{Similarly, randomized linear algebra approaches such as randomized SVD have been utilized for low-rank covariance approximation, but these approaches commonly require a specified rank approximation in advance \cite{Farchi2019}.}  Another alternative involved creating low-rank approximations of the estimated state vector and covariance matrix via the quasi-Newton Broyden–Fletcher–Goldfarb–Shanno (BFGS) method \cite{Auvinen2009} in the analysis update step of EnKF.

{More recently, Krylov subspace-based methods have been integrated into data assimilation in various manners to accomplish the goal of computational efficiency.  These methods naturally were incorporated into variational data assimilation methods \cite{Yaremchuk2009, Yaremchuk2015}.  At the same time, work has been explored utilizing Conjugate Gradient(CGD) and Lanczos iterations on the full EnKF analysis steps \cite{Bardsley2013, Bardsley2013-2}.  These advancements expanded into using Krylov subspace-based methods for efficient computations in Ensemble Square-Root Filters (ESRF) \cite{Steward2018,Armstrong2025}.  A more comprehensive list of advancements in iterative solvers for data assimilation methods is available \cite{Freitag2020}.  While these works demonstrate the computational efficiency gains of using Krylov subspace-based methods, they primarily utilize these iterative methods for computations on the full ensemble space, efficient covariance matrix representations, or under full square-root ensemble updates.}

In this paper, we present a novel variant of the EnKF algorithm that makes use of transposing a specific portion of the analysis step of the classical EnKF algorithm and utilizes CGD to approximate the inverse matrix within the Kalman Gain computation (CGD-EnKF).  This is distinct from the CGD-based EnKF algorithm presented in \cite{Bardsley2013, Bardsley2013-2}, which applies CGD to the entire analysis step; our proposed algorithm only applies CGD at the Kalman Gain computation.

The present work demonstrates that the proposed algorithm can reduce the computational time for highly dimensional dynamical systems and can be parallelized.  In fact, its computational efficiency is comparable to sEnKF while maintaining the same amount or improving in select cases, the state estimation approximation error.  Additionally, we establish a sharpened CGD error upper bound for clustered eigenvalue distributions of the covariance matrices.

The remainder of the paper is organized as follows.  Section \ref{sec:methods} discusses and highlights KF, deterministic EnKF, and sEnKF.  The proposed CGD-EnKF algorithm is derived and presented in Section \ref{sec:CGreformulation}.  Section \ref{sec:errorAnalysis} exhibits error analysis of the CGD approximation and the CGD-EnKF algorithm.  In Section \ref{sec:matrix_reform_alg}, we introduce a matrix reformulated algorithm, CGD-EnKF-Reduced, that is computationally efficient for a small ensemble but high observation dimensionality.  Section \ref{sec:experiments} presents numerical results to validate our proposed error bounds and solution validity and to demonstrate the computational efficiency and parallelizability of the proposed algorithms under different dynamical systems.  Section \ref{sec:conclusion} summarizes our results and discusses future work.

\section{Background}
\label{sec:methods}
In this section, we briefly review relevant Kalman filtering algorithms, including the classical Kalman Filter (KF), Ensemble Kalman Filter (EnKF), and serial EnKF (sEnKF), and establish the notation used throughout.

\subsection{State-Space Model}
\label{sec:SSM}
The Kalman filtering family of data assimilation algorithms maintains, at each discrete time step $k=0,1,\dots, T$ where $T$ is the final time step, a Gaussian distribution over the system’s (unknown) true state:
$$
\vect{\mathcal{X}}_k \sim \mathcal{N}(\vect{x}_k, P_k),
$$
where, {given $m$ is the state dimension}, $\vect{x}_k \in \mathbb{R}^m$ is the mean state estimate and $P_k \in \mathbb{R}^{m\times m}$ is the state covariance, a symmetric positive semi-definite matrix.

The standard linear–Gaussian state-space model is considered for the derivation of the Kalman filtering family of algorithms. This model consists of two equations: the state equation and the observation equation.  The state equation is given as
\begin{equation}
  \vect{\mathcal{X}}_k = F_{k-1}\, \vect{\mathcal{X}}_{k-1} + \vect{w}_{k-1}, 
  \label{eq:SSM1}
\end{equation}
where $F_{k-1}\in\mathbb{R}^{m\times m}$ is the linear model operator and the state model noise is given as $\vect{w}_{k-1}\sim\mathcal{N}(\vect{0}, Q_{k-1})$ with model uncertainty covariance $Q_{k-1}\in\mathbb{R}^{m\times m}$.

The observation equation is given as
\begin{equation}
  \vect{\mathcal{Y}}_k = H_k \vect{\mathcal{X}}_k + \vect{v}_k,
  \label{eq:SSM2}
\end{equation}
where, {given $p$ is the observation dimension or number of observation components}, $\vect{\mathcal{Y}}_k\in\mathbb{R}^p$ is the random observation vector and $H_k\in\mathbb{R}^{p\times m}$ is the linear observation operator. 
Moreover,  $\vect{v}_k\sim\mathcal{N}(\vect{0},R_k)$ with observation covariance $R_k\in\mathbb{R}^{p\times p}$, 
which quantifies the measurement error. Here, we emphasize that $R_k$ provides the degree of correlation between
observation components. Typically, $R_k$ is a non-diagonal matrix as most datasets have correlations between distinct
observation components. This corresponds to nonzero values for the entries $R_{(i,j),k}$ where $i\neq j$. Therefore,
it is often unrealistic to assume a diagonal $R_k$.
Upon collecting data at time $k$, we denote a realization of the random observation vector $\vect{\mathcal{Y}}_k$ by $\vect{y}_k \in \mathbb{R}^p$.

We assign a Gaussian prior $\vect{\mathcal{X}}_0\sim\mathcal{N}(\vect{x}_0,P_0)$, and $\{\vect{w}_k,\vect{v}_k\}$ are assumed to be white in time and mutually independent for all $k$.

\subsection{Kalman Filter}
\label{sec:KF}

In the linear--Gaussian state-space model of Section~\ref{sec:SSM}, the
conditional distribution of the state given observations up to time $k$ remains
Gaussian. We denote by $(\vect{x}_k^f, P_k^f)$ the forecast (prior) mean and
covariance of $\vect{\mathcal{X}}_k$ before assimilating $\vect{y}_k$, and by
$(\vect{x}_k^a, P_k^a)$ the analysis (posterior) mean and covariance of $\vect{\mathcal{X}}_k$ after
assimilating $\vect{y}_k$.

The Kalman Filter (KF) is a state estimation algorithm that comprises of a forecast step and an analysis step for each time step $k$~\cite{Kalman1960}. In the forecast step, the state equation~\eqref{eq:SSM1} is used to propagate the previous analysis state estimate $\vect{x}_{k-1}^a$ and previous analysis covariance $P_{k-1}^a$ forward in time to yield the forecast state estimate $\vect{x}_{k}^f$ and the forecast covariance $P_{k}^f$. Thus, we have
\begin{align}
\vect{x}_k^f &= F_{k-1} \vect{x}_{k-1}^a,    \label{alg:KF_forecast_1} \\
P_k^f &= F_{k-1} P_{k-1}^a F_{k-1}^T + Q_{k-1}. \label{alg:KF_forecast_2}
\end{align}

In the analysis step, the observation equation~\eqref{eq:SSM2} is utilized to determine how closely the observation operator $H_k$ maps the forecast state $\vect{x}_{k}^f$ to the observed data $\vect{y}_k$.

Then, the Kalman gain $K_k \in \mathbb{R}^{m \times p}$ is defined as
\begin{equation}
    K_k = P_k^f H_k^T S_k^{-1},
\end{equation}
where
\begin{equation}
    S_k = H_k P_k^f H_k^T + R_k
\end{equation}
is the innovation covariance, which determines how heavily the analysis step should weight $\vect{y}_k$ compared to the forecast state estimate $\vect{x}_k^f$.
We note that $K_k$ can be written in the form
\begin{equation}
K_k = \begin{bmatrix}
        K_k\hat{\vect{e}}_1 & K_k\hat{\vect{e}}_2 & \dots  & K_k\hat{\vect{e}}_p
    \end{bmatrix},
\end{equation}
where $\hat{\vect{e}}_j \in \mathbb{R}^p$ for $j=1,2,...,p$ are the standard basis column vectors.

Then, $K_k$ and the innovation are used to update $\vect{x}_{k}^f$ and $P_{k}^f$ to produce the analysis estimates $\vect{x}_{k}^a$ and $P_{k}^a$ as follows:
\begin{align}
\vect{x}_k^a &= \vect{x}_k^f + K_k (\vect{y}_k - H_k \vect{x}_k^f), \label{alg:KF_analysis_1} \\
P_k^a &= (I - K_k H_k)\, P_k^f. \label{alg:KF_analysis_2}
\end{align}

For each time step $k$, the forecast step and, when observations are present, the analysis step are applied iteratively, assuming that the initial data
$\vect{x}_0^a, P_0^a$ (with $\vect{x}_0^a = \vect{x}_0$, $P_0^a = P_0$) and the
sequences $\{\vect{y}_k\}$, $\{F_{k-1}\}$, $\{H_k\}$, $\{Q_{k-1}\}$, and
$\{R_k\}$ for $k = 1,\dots,T$ are provided.

For the remainder of the work, unless otherwise specified or when $k$ subscripts would differ, we will suppress the time step subscript notation $k$.

\subsection{Ensemble Kalman Filter}
\label{sec:EnKF}
For linear--Gaussian problems, the Kalman filter (KF) works effectively,
providing accurate state estimates~\cite{Kalman1961}. For nonlinear
problems, however, KF may yield state estimates that diverge significantly
from the true state~\cite{Fitzgerald1971}. To address filter divergence occurring from the nonlinearity, the Ensemble Kalman Filter (EnKF) was
introduced~\cite{evensen2003ensemble}.

At time step $k$, the EnKF represents the forecast distribution
$\vect{\mathcal{X}}^f$ by {an ensemble size of $N$ state vectors},
\begin{equation}
    X^f = \begin{bmatrix}
        \vect{x}_{1,k}^f & \vect{x}_{2,k}^f & \dots  & \vect{x}_{N,k}^f
    \end{bmatrix} \in \mathbb{R}^{m \times N},
\end{equation}
where $\vect{x}_{i,k}^f := \vect{x}_{i}^f$ denotes the $i$-th ensemble member.

In the forecast step, each analysis ensemble member $\vect{x}^a_{i,k-1}$ is
propagated individually under the model operator $F_{k-1}$:
\begin{equation}
  \vect{x}_{i,k}^f = F_{k-1}\,\vect{x}^a_{i,k-1}, \qquad i=1,\dots,N.
\end{equation}
The ensemble mean $\vect{\bar{x}}^f$ and covariance $P^f$ at time $k$
are then computed as
\begin{align}
\vect{\bar{x}}^f &= \frac{1}{N} \sum_{i=1}^{N} \vect{x}_{i}^f, \\
P^f &= \frac{1}{N-1} \sum_{i=1}^{N}
(\vect{x}_{i}^f - \vect{\bar{x}}^f)(\vect{x}_{i}^f - \vect{\bar{x}}^f)^T.
\end{align}

The analysis step follows the KF structure but uses ensemble-based
covariances. We form the predicted observation ensemble
$Z = H X^f \in \mathbb{R}^{p \times N}$ in block form:
\begin{equation}
Z = \begin{bmatrix}
        \vect{z}_{1} & \vect{z}_{2} & \dots  & \vect{z}_{N}
    \end{bmatrix}
    = \begin{bmatrix}
        z_{(1,1)} & \dots & z_{(1,N)} \\
        z_{(2,1)} & \dots & z_{(2,N)} \\
        \vdots      & \ddots & \vdots     \\
        z_{(p,1)} & \dots & z_{(p,N)}
    \end{bmatrix},
\end{equation}
where $\vect{z}_{i} \in \mathbb{R}^p$ is the predicted observation corresponding to
$\vect{x}_{i}^f$. The ensemble mean of the predicted observations is
\begin{equation}
  \vect{\bar{z}} = \frac{1}{N} \sum_{i=1}^{N} \vect{z}_{i}.
\end{equation}
For further use, we define the centered observation anomaly matrix $\widehat{Z} \in \mathbb{R}^{p \times N}$ as                     
  \begin{equation}                                          
    \widehat{Z} = Z - \vect{\bar{z}}\,\mathbbm{1}_N^T,                                    
    \label{eq:Zhat_def}
  \end{equation}                                                                                                                                                                                                                                                              
  whose $i$-th column is $\vect{z}_i - \vect{\bar{z}}$. 

We then compute the Kalman gain $K \in \mathbb{R}^{m \times p}$ from the
ensemble-based cross- and auto-covariances $\Psi_x \in \mathbb{R}^{m \times p}$ and $\Psi_z \in \mathbb{R}^{p \times p}$:
\begin{align}
    K &= \Psi_{x} \,\Psi_{z}^{-1}, \label{eq:K_EnKF} \\
    \Psi_{x} &=
    \frac{1}{N-1} \sum_{i=1}^{N}
    (\vect{x}_{i}^f - \vect{\bar{x}}^f)
    (\vect{z}_{i} - \vect{\bar{z}})^T, \\   
    \Psi_{z} &=
    \frac{1}{N-1} \sum_{i=1}^{N}
    (\vect{z}_{i} - \vect{\bar{z}})
    (\vect{z}_{i} - \vect{\bar{z}})^T + R
    = \frac{1}{N-1}\,\widehat{Z}\widehat{Z}^T + R.
    \label{eq:Psi_zz}    
\end{align}
Finally, we update the ensemble using the observation $\vect{y}$:
\begin{equation}
    X^a = X^f + K (Y - Z),
\end{equation}
where $Y=\bm{y}\mathbbm{1}_N^T \in \mathbb{R}^{p \times N}$ with \(\mathbbm{1}_N=(\underbrace{1,\ldots,1}_{N})^T\).

\begin{remark}
The described EnKF algorithm is known as the deterministic EnKF algorithm.  This is one of two main classifications of EnKF algorithms: stochastic \cite{Houtemaker1998} or deterministic~\cite{Whitaker2002}.  Stochastic EnKF relies on perturbing the observation $\vect{y}$ to create an observation ensemble where each member corresponds to one of the state estimate members. Deterministic EnKF relies on utilizing the same exact observation value for each observation ensemble member to shift all ensemble member state estimates. Deterministic EnKF is typically more accurate than stochastic EnKF for a small ensemble size.  However, stochastic EnKF produces state estimates that are closer to the true state in non-Gaussian distributions.  Since our work focuses on Gaussian distributions and aims to produce computationally efficient, accurate solutions, we consider the deterministic EnKF.
\end{remark}

\subsection{Serial Ensemble Kalman Filter}
\label{sec:sEnKF}

EnKF vastly improved accuracy for moderately nonlinear data assimilation problems, but the ensemble-based approach adds computational complexity with increasing dimensionality.  In particular,  $ K =\Psi_{x} \Psi_{z}^{-1}  \in \mathbb{R}^{m \times p}$ in \eqref{eq:K_EnKF} contains a typically dense matrix inversion that becomes computationally prohibitive when the observation dimension $p$ is large~\cite{Houtekamer2001}.  

To address this, the state-of-the-art serial EnKF (sEnKF) reformulates \eqref{eq:K_EnKF} by assimilating each observation component in a sequential manner at a given step $k$ \cite{Houtekamer2001}.  Therefore, instead of directly inverting the $\Psi_{z} \in \mathbb{R}^{p \times p}$, sEnKF performs $p$ rank-1 scalar inversion updates:
\begin{equation}
    K\hat{\vect{e}}_j = \vect{\psi}_{j,x} /  \psi_{j,z},
\end{equation}
from $j=1, \cdots, p$ where
\begin{align}
    \vect{\psi}_{j,x} &= \frac{1}{N-1} \sum_{i=1}^{N} (\vect{x}_{i}^f - \vect{\bar{x}}^f)(z_{(j,i)} - \bar{z}_{j}), \\
    \psi_{j,z} &= \frac{1}{N-1} \sum_{i=1}^{N} (z_{(j,i)} - \bar{z}_{j})^2 + R_{(j,j)}.
    \label{Obs_Cov_sEnKF}
\end{align}

This formulation allows for the computation of Kalman Gain vectors that can be independently computed based on each observation component.  Therefore, further computational time reduction is possible through parallelization. 
Notably, \eqref{Obs_Cov_sEnKF} shows that $R$ must be diagonal, otherwise sEnKF will neglect cross-covariances between observations.  Thus, $R$ having nonzero off-diagonal entries may significantly impact the accuracy of the state estimations through sEnKF. The comparisons of the analysis step between EnKF and sEnKF algorithm are presented in Figure \ref{fig:comparison}.

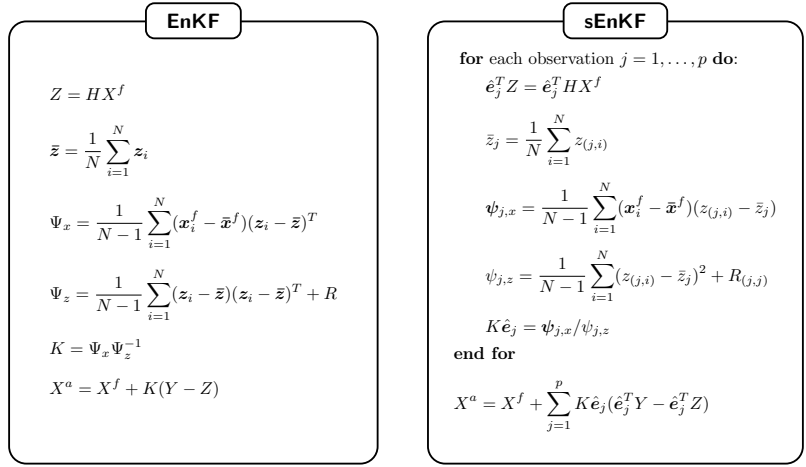
\begin{figure}[htbp]
  \centering
  \tikzset{
    algo-box/.style={
      rectangle,
      draw=black,
      fill=white,
      line width=0.7pt,
      rounded corners=1.5ex,
      minimum width=7.5cm,
      minimum height=9cm,
      align=left,
      inner xsep=15pt,
      inner ysep=15pt
    },
    title-badge/.style={
      font=\sffamily\bfseries\large,
      draw=black,
      fill=white,
      text=black,
      line width=0.7pt,
      rounded corners=0.75ex,
      inner xsep=12pt,
      inner ysep=6pt
    }
  }

  \begin{tikzpicture}[node distance=1cm, scale=0.65, transform shape]

    \node[algo-box] (enkf) {
      \vspace{0.2cm}
      $\begin{aligned}
        & Z = H X^f \\[1.5ex]
        & \vect{\bar{z}} = \dfrac{1}{N} \sum_{i=1}^{N} \vect{z}_{i} \\[1.5ex]
        & \Psi_{x} = \dfrac{1}{N-1} \sum_{i=1}^{N} (\vect{x}_{i}^f - \vect{\bar{x}}^f)(\vect{z}_{i} - \vect{\bar{z}})^T \\[1.5ex]
        & \Psi_{z} = \dfrac{1}{N-1} \sum_{i=1}^{N} (\vect{z}_{i} - \vect{\bar{z}})(\vect{z}_{i} - \vect{\bar{z}})^T + R \\[1.5ex]
        & K = \Psi_{x} \Psi_{z}^{-1} \\[1.5ex]
        & X^a = X^f + K (Y -Z)
      \end{aligned}$
    };
    \node[title-badge] at (enkf.north) {EnKF};

    \node[algo-box, right=of enkf] (serial) {
      \vspace{0.2cm}
      \textbf{for} each observation $j = 1, \dots, p$ \textbf{do}: \\[1ex]
      \hspace*{1.5em} $\begin{aligned}
        & \hat{\vect{e}}_j^T Z = \hat{\vect{e}}_j^T H X^f \\[1ex]
        & \bar{z}_{j} = \dfrac{1}{N} \sum_{i=1}^{N} z_{(j,i)} \\[1ex]
        & \vect{\psi}_{j,x} = \dfrac{1}{N-1} \sum_{i=1}^{N} (\vect{x}_{i}^f - \vect{\bar{x}}^f)(z_{(j,i)} - \bar{z}_{j}) \\[1ex]
        & \psi_{j,z} = \dfrac{1}{N-1} \sum_{i=1}^{N} (z_{(j,i)} - \bar{z}_{j})^2 + R_{(j,j)} \\[1ex]
        & K\hat{\vect{e}}_j = \vect{\psi}_{j,x} / \psi_{j,z}
      \end{aligned}$ \\[1ex]
      \textbf{end for} \\[2.5ex]
      $\displaystyle X^{a} = X^{f} + \sum_{j=1}^{p} K\hat{\vect{e}}_j (\hat{\vect{e}}_j^T Y - \hat{\vect{e}}_j^T Z)$
    };
    \node[title-badge] at (serial.north) {sEnKF};

  \end{tikzpicture}
  
  \vspace{0.5cm}
  \caption{Comparison of the $k$-th analysis step between the EnKF and sEnKF.}
  \label{fig:comparison}
\end{figure}

\section{Reformulation of EnKF based on Conjugate Gradient}
\label{sec:CGreformulation}

As previously noted, the inverse computation in the EnKF algorithm as seen in \eqref{eq:K_EnKF} where the typically dense invertible matrix is given by \eqref{eq:Psi_zz} can be computationally expensive, particularly for problems with highly dimensional observation dimension (large $p$).  The state-of-the-art sEnKF, as outlined in Section \ref{sec:sEnKF}, significantly reduces the computational expense, but adds a restrictive and often non-realistic assumption that $R$ is diagonal.  Therefore, the cross-covariance between distinct observation components is not considered in the state estimation.  

In this section, we present a reformulation of \eqref{eq:K_EnKF} with the classical Conjugate Gradient (CGD) iterative method to approximate $K$. 
Our proposed approach reduces the computational complexity compared to the EnKF algorithm while relaxing the sEnKF condition that $R$ be diagonal.  We also note that, similar to sEnKF, our proposed method is parallelizable.

\subsection{Reformulating The Inverse Computation in EnKF with CGD}
\label{sec:inverseComp}

For this reformulation, we utilize the Conjugate Gradient~(CGD) method \cite{Hestenes1952}.  
CGD is an iterative method for solving $A\vect{x}=\vect{b}$ problems where $A \in {\mathbb{R}^{p \times p}}$ is symmetric positive definite and $\vect{b}\in {\mathbb{R}^p}$ that has a block CGD formulation solving $AX=B$ where $B\in {\mathbb{R}^{p\times m}}$.  This iterative method is known to converge in at most {$p$} steps.  Practically, this method is effective for high dimensional systems where solving $X=A^{-1}B$ has a significant computational expense.

The inverted matrix term from the computation in \eqref{eq:K_EnKF} multiplies from the right.  Therefore, this equation does not yet resemble the standard block linear equation $AX=B$, which traditionally requires the inverted term to be multiplied from the left.  However, by transposing \eqref{eq:K_EnKF} and noting that $\Psi_{z}$ is symmetric, we obtain
\begin{equation}
    K^T= \Psi_{z}^{-1} \Psi_{x}^T \in \mathbb{R}^{p \times m}.
\label{eq:K_KF_T}
\end{equation}

Rewriting now yields the standard block linear equation:
\begin{equation}
\Psi_{z} K^T= \Psi_{x}^T,
\label{eq:K_KF_T_rewrite}
\end{equation}
where
\begin{equation}
A=\Psi_{z}, \quad X = K^T, \quad B = \Psi_{x}^T.
\end{equation}
Note that $\Psi_{z}$ is symmetric positive definite, so \eqref{eq:K_KF_T_rewrite} can be solved with block CGD, yielding an approximation of $K^T$.

\begin{remark}
A similar reformulation of the $K$ computation can be implemented for KF algorithm.
\end{remark}

\subsection{Reformulated EnKF Algorithm with CGD}
\label{sec:reformulatedAlgorithm}

Given the reformulation of the inverse computation, we now present the novel CGD-EnKF algorithm. 
The comparison of the analysis step with the EnKF method is presented in Figure \ref{fig:comparison_CG}.
This algorithm provides an alternative to sEnKF that likewise avoids the significant numerical expense of forming a dense inverse matrix.  The block CGD makes the analysis step of CGD-EnKF parallelizable, making the algorithm further comparable to sEnKF in terms of efficiency.  In addition, for both methods, there is a storage-efficiency where a full inverse matrix does not need to be stored.  Therefore, CGD-EnKF, in terms of computational expense, is comparable to the state-of-the-art Kalman filtering-based algorithm.

The primary difference between CGD-EnKF and sEnKF is seen in state estimation accuracy.  
Note that  CGD-EnKF makes use of the whole $R$, in \eqref{eq:Psi_zz}, whereas sEnKF is restricted to consider only diagonal entries of $R$,   $R_{(j,j)}$ in \eqref{Obs_Cov_sEnKF}, neglecting possible off-diagonal entries.  Therefore, in the case where observation components have cross-covariance, CGD-EnKF is able to account for these in the state estimation computation, whereas sEnKF assumes these cross-covariances are $0$.

Thus, in this more realistic, specific case of considering nonzero cross-covariance, CGD-EnKF may produce a more accurate state estimation approximation than the sEnKF.

 \subsection{ Parallel Implementation of CGD-EnKF}\label{sec:parallel_cgd}
A key advantage of the CGD-EnKF algorithm is its natural suitability
  for parallel computation.  We adopt a row-block
  distribution of the observation dimension $p$ across $\ n_p$
  processors.  Specifically, each processor $q$ owns a contiguous block of
  $p_q \approx p/n_p$ rows of every distributed matrix and
  vector.  We now will discuss the computational cost of a parallelized CGD-EnKF and discuss when it is favorable to utilize.

 Recall from Figure~\ref{fig:comparison_CG} that the CGD-EnKF analysis step requires solving $m$ independent linear systems via CGD since these $m$ systems share the same coefficient matrix $\Psi_z$ but have independent right-hand sides.

 The dominant cost in each conjugate gradient iteration is the distributed matrix-vector product $\Psi_z \vect{d}$ where $\vect{d} \in \mathbb{R}^p$ is the search direction at the given CGD iteration.
        Regarding~\eqref{eq:Psi_zz}, this product decomposes as
    \begin{equation}\label{eq:matvec_decomp}
      \Psi_z \vect{d}
      = \frac{1}{N-1}\,\widehat{Z}\bigl(\widehat{Z}^T \vect{d}\bigr) + R\,\vect{d}.
    \end{equation}
    Under the row-block distribution, the computation in each processor proceeds in three stages:
    \begin{enumerate}
      \item Each processor computes its local contribution to
        $\widehat{Z}^T \vect{d} \in \mathbb{R}^N$.
      \item Each processor computes its local $p_q$ rows of
        $\frac{1}{N-1}\,\widehat{Z}\bigl(\widehat{Z}^T \vect{d}\bigr)$.
      \item Each processor computes its local rows of $R\,\vect{d}$. 
    \end{enumerate}
    Stages~1 and~2 each cost $\mathcal{O}(pN/n_p)$ per processor.  Stage 3 costs $\mathcal{O}(p/n_p)$ for sparse $R$ (e.g.\ diagonal or tridiagonal) and $\mathcal{O}(p^2/n_p)$ for dense $R$.
    When $R$ is sparse, stage~3 is dominated by stages~1--2,
    and each CGD iteration for each of the $m$ systems costs $\mathcal{O}(pN/n_p)$ per processor,
    yielding a total parallel cost of $\mathcal{O}(mrpN/n_p)$.
    When $R$ is dense, stage~3 costs $\mathcal{O}(p^2/n_p)$ which may contribute a significant amount of computational expense compared with the first two stages so each CGD iteration for each of the $m$ systems costs $\mathcal{O}(p(N+p)/n_p)$ with a total parallel cost of $\mathcal{O}(rmp(N+p)/ n_p)$ {given $r$ CGD iterations.  To determine an appropriate number of CGD iterations, one must set a convergence criteria.  One practice given a known exact solution for $K^T$ and the way in which this paper implements such criteria in the numerical experiments, unless given a fixed number of iterations, is through measuring the absolute error $\vect{e}_r$ under a chosen norm $\|\cdot\|$ of the transposed Kalman Gain approximation at the $r$-th iteration $(K^T)_r$ defined as 
    \begin{equation}
        \|\vect{e}_r\| = \|K^T - (K^T)_r\|  < \epsilon
    \end{equation}
    until some defined error tolerance $\epsilon$ is reached.  As an alternative for an unknown Kalman gain solution, one could easily instead measure the block residual error up to the CGD iteration that reaches $\epsilon$
    \begin{equation}
         \|\Psi_x^T - \Psi_z (K^T)_r\|  < \epsilon.
    \end{equation}}

    Table \ref{all_method} provides a comparison between the methods in terms of computational cost and accuracy given large $p$.  Notice that if $m\ll p$, $N\ll p$, or $n_p$ is large, then EnKF is exceedingly more expensive compared with sEnKF and CGD-EnKF.  CGD-EnKF is comparable to sEnKF in terms of computational cost if $R$ is sparse and CGD only requires a small number of $r$ iterations to converge.  If $R$ is sparse but non-diagonal with small $r$ required for CGD convergence, then CGD-EnKF will have a guaranteed high accuracy compared to sEnKF while maintaining similar computational cost.

 \begin{table}[!h]                                           
    \centering                           
    \footnotesize                                            
    \renewcommand{\arraystretch}{1.3}                      
    \begin{tabularx}{0.95\textwidth}{@{} l l l l X @{}} 
    \toprule                                                
    \textbf{Method} & \textbf{Inversion Type} & \textbf{Complexity} & \textbf{Accuracy} \\
    \midrule
    EnKF & Full matrix inverse 
      & $\mathcal{O}(p^3)$
      & Exact  \\
    sEnKF (diagonal $R$) & Scalar inverse
      & $\mathcal{O}(mpN/ n_p)$
      & Exact \\
    sEnKF (non-diagonal $R$) & Scalar inverse
      & $\mathcal{O}(mpN/ n_p)$
      & Potentially Low \\
    CGD-EnKF (sparse $R$) & Iterative
      & $\mathcal{O}(rmpN/ n_p)$
      & Converges with $r$ \\
    CGD-EnKF (dense $R$) & Iterative 
      & $\mathcal{O}(rmp(N+p)/ n_p)$
      & Converges with $r$ \\
    \bottomrule
    \end{tabularx}
    \caption{Comparison of the analysis step of EnKF-based algorithms on computational complexity and accuracy given large $p$ and $n_p$ processors available for parallel implementation capabilities.}
    \label{all_method}
  \end{table}

\begin{figure}[htbp]
  \centering
  
  \tikzset{
    algo-box/.style={
      rectangle,
      draw=black,
      fill=white,
      line width=0.7pt,
      rounded corners=1.5ex,
      minimum width=7.5cm,
      minimum height=9.5cm,
      align=left,
      inner xsep=15pt,
      inner ysep=15pt
    },
    title-badge/.style={
      font=\sffamily\bfseries\large,
      draw=black,
      fill=white,
      text=black,
      line width=0.7pt,
      rounded corners=0.75ex,
      inner xsep=12pt,
      inner ysep=6pt
    }
  }

\begin{tikzpicture}[node distance=1cm, scale=0.65, transform shape]

    \node[algo-box] (enkf) {
      \vspace{0.2cm}
      $\begin{aligned}
        & Z = H X^f \\[1.5ex]
        & \vect{\bar{z}} = \dfrac{1}{N} \sum_{i=1}^{N} \vect{z}_{i} \\[1.5ex]
        & \Psi_{x} = \dfrac{1}{N-1} \sum_{i=1}^{N} (\vect{x}_{i}^f - \vect{\bar{x}}^f)(\vect{z}_{i} - \vect{\bar{z}})^T \\[1.5ex]
        & \Psi_{z} = \dfrac{1}{N-1} \sum_{i=1}^{N} (\vect{z}_{i} - \vect{\bar{z}})(\vect{z}_{i} - \vect{\bar{z}})^T + R \\[1.5ex]
        & K = \Psi_{x} \Psi_{z}^{-1} \\[1.5ex]
        & X^a = X^f + K (Y -Z)
      \end{aligned}$
    };
    \node[title-badge] at (enkf.north) {EnKF};

    \node[algo-box, right=of enkf] (cg) {
      \vspace{0.2cm}
      $\begin{aligned}
        & Z = H X^f \\[1.5ex]
        & \vect{\bar{z}} = \dfrac{1}{N} \sum_{i=1}^{N} \vect{z}_{i} \\[1.5ex]
        & \Psi_{x} = \dfrac{1}{N-1} \sum_{i=1}^{N} (\vect{x}_{i}^f - \vect{\bar{x}}^f)(\vect{z}_{i} - \vect{\bar{z}})^T \\[1.5ex]
        & \Psi_{z} = \dfrac{1}{N-1} \sum_{i=1}^{N} (\vect{z}_{i} - \vect{\bar{z}})(\vect{z}_{i} - \vect{\bar{z}})^T + R
      \end{aligned}$ \\[2ex]
      \textbf{for} each column $i = 1, \dots, m$ of $K^T \in \mathbb{R}^{p \times m}$ \textbf{do}: \\[1ex]
      \hspace*{1.5em} Solve $\Psi_{z} (K^T \hat{\vect{e}}_i) = \Psi^T_{x}\hat{\vect{e}}_i$ by CGD \\[1ex]
      \textbf{end for} \\[2ex]
      $\begin{aligned}
        & K = (K^T)^T \\[1.5ex]
        & X^a = X^f + K (Y - Z)
      \end{aligned}$
    };
    \node[title-badge] at (cg.north) {CGD-EnKF};

  \end{tikzpicture}
  
  \vspace{0.5cm}
  \caption{Comparison of the analysis step between EnKF and CGD-EnKF.}
  \label{fig:comparison_CG}
\end{figure}
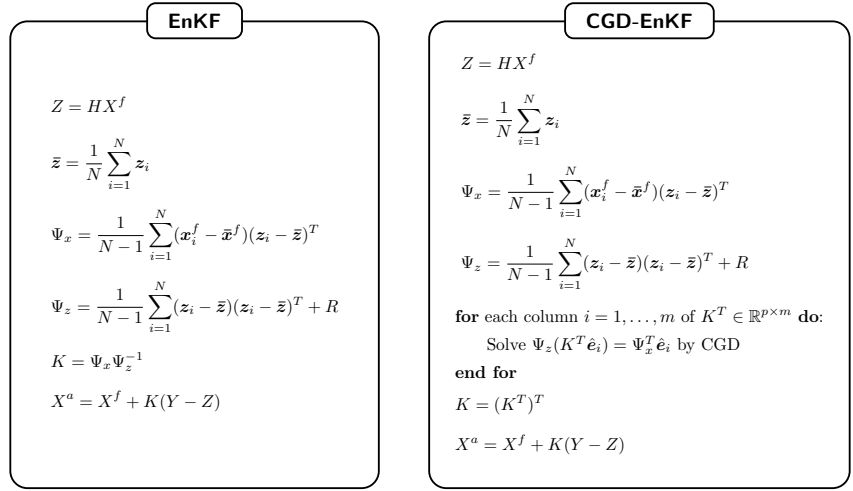

\section{Error Analyses}
\label{sec:errorAnalysis}

In this section, we provide the error analysis for the provided CGD algorithm. In particular, we bound the error from the approximation of the $\Psi_{z}^{-1}$ through CGD and demonstrate that under certain conditions, the $K^T$ approximation will converge to the true $K^T$ at the given $k$-th time step.  
Thus, the state estimations obtained from CGD-EnKF will converge to those from EnKF.  

The analysis that follows is based on known error and convergence theory for CGD \cite{Shewchuk1994} as well as prior analysis work for when CGD is applied to the entire analysis step \cite{Bardsley2013,Bardsley2013-2} rather than only the Kalman Gain computation from the analysis step as seen in our CGD-EnKF.

We define the CGD error as

\begin{equation}
    \vect{e}_r = K^T - (K^T)_r = \Psi_{z}^{-1}\Psi_{x}^T - (\Psi_{z}^{-1})_r\Psi_{x}^T
    \label{eq:e_r_def}
\end{equation}
where $(\cdot)_r$ is the approximation of a given quantity at the $r$-th CGD iteration.  Using this definition, the following proposition states the classical CGD upper error bound established in \cite{Shewchuk1994}.
\begin{proposition}
Let $r<p$.  Let the ordered eigenvalues of 
    $\Psi_{z}$ be $\lambda_1 \geq \lambda_2 \geq ... \geq \lambda_{r-1} \geq \lambda_r \geq \lambda_{r+1} \geq ... \geq \lambda_{p-1} \geq \lambda_p$. Then, there exists the following error bound on a rank $r$ CGD approximation of $\Psi_{z}^{-1}$:
    \begin{equation}
    ||\Psi_{z}^{-1} - \Sigma_r D_r^{-1} \Sigma_r^T||_{\Psi_{z}} \leq 2(\frac{\sqrt{\operatorname{cond}(\Psi_{z})} - 1}{\sqrt{\operatorname{cond}(\Psi_{z})} +1})^r \| \vect{e}_0\|_{\Psi_{z}},
    \label{eqn:prop1}
    \end{equation}
where $\operatorname{cond}(\Psi_{z})$ is the condition number of $\Psi_{z}$, $\Sigma_r \in \mathbb{R}^{p \times r}$ contains the first $r$ eigenvectors of $\Psi_{z}$, $D_r^{-1}$ is a diagonal matrix with the first $r$ eigenvalues along the diagonal, and $\vect{e}_0$ is the initial CGD approximation error.
\label{prop:1}
\end{proposition}

Using the approximation of $\Psi_{z}$ and the worst-case upper error bound, we now show that our $K^T$ approximation will converge to the true $K^T$ at the given $k$-th time step.

\begin{proposition}
    As $r \rightarrow \infty$, the error of the CGD approximation goes to $0$ and $(\Psi_{z}^{-1})_r \rightarrow \Psi_{z}^{-1}$.  Additionally, this indicates that CGD-EnKF will converge to the state estimation solution ensemble $X^a$ of EnKF .
    \label{prop:2}
\end{proposition}

\begin{proof}
With the given error bound from \eqref{eqn:prop1}:
\begin{equation}
\| \vect{e}_r\|_{\Psi_{z}} \leq 2(\frac{\sqrt{\operatorname{cond}(\Psi_{z})} - 1}{\sqrt{\operatorname{cond}(\Psi_{z})} +1})^r \| \vect{e}_0\|_{\Psi_{z}}, 
\end{equation}
we note that     
$(\frac{\sqrt{\operatorname{cond}(\Psi_{z})} - 1}{\sqrt{\operatorname{cond}(\Psi_{z})} +1})^r \rightarrow 0$ when $r \rightarrow \infty$ since $(\frac{\sqrt{\operatorname{cond}(\Psi_{z})} - 1}{\sqrt{\operatorname{cond}(\Psi_{z})} +1}) < 1$. 
Therefore, the upper error bound goes to $0$ as $r\rightarrow \infty$:
$\| \vect{e}_{r\rightarrow \infty}\|_{\Psi_{z}} \leq 0$,

By the definition of $\| \vect{e}_r\|_{\Psi_{z}}$ in \eqref{eq:e_r_def}, we obtain
\begin{equation}
\Psi_{z}^{-1}\Psi_{x}^T - (\Psi_{z}^{-1})_{r\rightarrow \infty}\Psi_{x}^T = 0.    
\end{equation}
This indicates that $(\Psi_{z}^{-1})_r \rightarrow \Psi_{z}^{-1}$ when $r \rightarrow \infty$ and likewise  $(K^T)_r \rightarrow K^T$.
By the definition of a transpose $K = (K^T)^T$, meaning that if $r\rightarrow \infty$, CGD-EnKF produces the same Kalman Gain $K$ as EnKF after being transposed.  Therefore, the state estimations of the CGD-EnKF solution will converge to the solution $X^a$ of EnKF as $r \rightarrow \infty$.
\end{proof}

Notably, the classical CGD error bound assumes the worst case of uniformly distributed eigenvalues.  However, in practice, covariance matrices often have clustered and extreme eigenvalues.  To account for this, we aim to derive a sharper error bound in the covariance matrix context using the generally proposed bound in \cite{Zhu2008}.  We will define a set of eigenvalues $\sigma_0(\Psi_{z})$ to contain unique extremal eigenvalues of $\Psi_{z}$ and a set of eigenvalues $\sigma_1(\Psi_{z})$ to contain any other unique eigenvalues of $\Psi_{z}$.

\begin{proposition}
Let $1\leq r<p$.  Let the ordered eigenvalues of 
    $\Psi_{z}$ be $\lambda_1 \geq \lambda_2 \geq ...\geq \lambda_{b-1} \geq \lambda_b \geq \lambda_{b+1} \geq  ... \geq \lambda_{a-1} \geq \lambda_a \geq \lambda_{a+1} \geq  ... \geq \lambda_{p-1} \geq \lambda_p$. Let there be $l$ unique eigenvalues of $\Psi_{z}$ that are less than $\lambda_a$ or greater than $\lambda_b$ denoted $\mu_i$ for $i=1,...,l$ in a set $\sigma_0(\Psi_{z})$.  Additionally, let all other unique eigenvalues of $\Psi_{z}$ ranging from $\lambda_a$ to $\lambda_b$ be in a set $\sigma_1(\Psi_{z})$.  Then, if $r \geq l$, there exists the following error bound on a rank $r$ CGD approximation of $\Psi_{z}^{-1}$:
    \begin{align}
    ||\Psi_{z}^{-1} - \Sigma_r D_r^{-1} \Sigma_r^T||_{\Psi_{z}} \leq 2C(\frac{\sqrt{\operatorname{cond}(\Psi_{z})_{eff}} - 1}{\sqrt{\operatorname{cond}(\Psi_{z})_{eff}} +1})^{r-l} \| \vect{e}_0\|_{\Psi_{z}}, \\
    C = \max\limits_{\substack{\lambda \in \sigma_1(\Psi_{z})}} \prod_{i=1}^{l} | 1-\frac{\lambda}{\mu_i}|
    \end{align}
where $\operatorname{cond}(\Psi_{z})_{eff}$ is the effective condition number of $\Psi_{z}$ defined as $\frac{\lambda_b}{\lambda_a}$, $\Sigma_r \in \mathbb{R}^{p \times r}$ contains the first $r$ eigenvectors of $\Psi_{z}$, $D_r^{-1}$ is a diagonal matrix with the first $r$ eigenvalues along the diagonal, and $\vect{e}_0$ is the initial CGD approximation error.
\label{prop:3}
\end{proposition}

\begin{proof}
Since $\Psi_{z}$ is symmetric positive definite, it is diagonalizable into a matrix $D_p$ by its eigenmatrix.  Therefore,
\begin{equation}
D_p = \begin{bmatrix}
 D_r & 0\\
 0 & D_q
\end{bmatrix} = \begin{bmatrix}
 \Sigma_r^T \Psi_{z} \Sigma_r & 0\\
 0 & \Sigma_q^T \Psi_{z}^T \Sigma_q
\end{bmatrix} 
= \Sigma_p^T \Psi_{z} \Sigma_p,
\label{eq:diagonalization}
\end{equation}
where we define $\Sigma_p \in \mathbb{R}^{p \times p}$ as all eigenvectors of $\Psi_{z}$ in the columns, $\Sigma_r \in \mathbb{R}^{p \times r}$ as the first $r$ eigenvectors of $\Psi_{z}$ in the columns, and $\Sigma_q \in \mathbb{R}^{p \times (p-r)}$ as the last $p-r$ eigenvectors of $\Psi_{z}$ in the columns.

Since $D_p$ is diagonal with non-zero entries, $D_p$ is invertible, so we can apply the inverse to both sides of \eqref{eq:diagonalization}:
\begin{equation}
    \Psi_{z}^{-1} = \Sigma_p D_p^{-1} \Sigma_p^T = \Sigma_r D_r^{-1} \Sigma_r^T + \Sigma_q D_q^{-1} \Sigma_q^T.
\end{equation}
Therefore, we can write an $r$-rank approximation of $\Psi_{z}^{-1}$ as
\begin{equation}
(\Psi_{z}^{-1})_r = \Sigma_r D_r^{-1} \Sigma_r^T,
\end{equation}
and the corresponding error of the $r$-rank approximation solution to $\Psi_{z}K^T = \Psi_{x}^T$ is
\begin{equation}
    \vect{e}_r = K^T - (K^T)_r = \Psi_{z}^{-1}\Psi_{x}^T - (\Psi_{z}^{-1})_r\Psi_{x}^T
    \label{eq:e_r_def}
\end{equation}

Since we are utilizing CGD to make the $r$-rank approximation, we can rather rewrite the approximation error as 
\begin{equation}
    \vect{e}_r = \vect{e}_0 + \text{span}\{\Psi_{z}\vect{e}_0, (\Psi_{z})^2\vect{e}_0,...,(\Psi_{z})^r\vect{e}_0\}
    = (I + \sum_{j=1}^{r} \omega_j(\Psi_{z})^j)\vect{e}_0,
    \label{e_r_def_rewrite}
\end{equation}
where $\omega_j$ are the weights to cover the generated Krylov Space and $\vect{e}_0$ is the initial CGD approximation error.  
We can rewrite \eqref{e_r_def_rewrite} in terms of a matrix polynomial expression:
\begin{equation}
\vect{e}_r = \mathcal{P}_r(\Psi_{z})\vect{e}_0,
\label{e_r_mat_poly}
\end{equation}
where $\mathcal{P}_r(\Psi_{z}) = I + \sum_{j=1}^{r} \omega_j(\Psi_{z})^j$.

With the given eigenvalue $\lambda$ and its corresponding eigenvector $\vect{v}$ we obtain,
\begin{equation}
\mathcal{P}_r(\Psi_{z})\vect{v}=\mathcal{P}_r(\lambda)\vect{v},
\end{equation}
and we can rewrite \eqref{e_r_mat_poly} as
\begin{equation}
\vect{e}_r= \sum_{j=1}^{r}\mathcal{P}_r(\lambda_j)\vect{e}_0.
\end{equation}

Applying the definition of the energy norm of $\Psi_{z}$, we get
\begin{equation}
    \| \vect{e}_r\|_{\Psi_{z}}^2 = \sum_{j=1}^{r} [\mathcal{P}_r(\lambda_j)]^2 \| \vect{e}_0\|_{\Psi_{z}}^2.
\end{equation}

To begin constructing an upper bound, we want to choose the eigenvalue that maximizes the right side.  This is dependent upon the chosen polynomials $\mathcal{P}_r$:
\begin{equation}
    \| \vect{e}_r\|_{\Psi_{z}}^2 \leq \max\limits_{\substack{j=1,...,r}} [\mathcal{P}_r(\lambda_j)]^2 \| \vect{e}_0\|_{\Psi_{z}}^2.
    \label{eq:constr_upper_bound}
\end{equation}

    For the eigenvalues $\mu_i$ for $i=1,...,l$ in the set $\sigma_0(\Psi_{z})$, a matrix polynomial $\mathcal{R}_l(\lambda)$ can be defined.  We specifically choose a matrix polynomial that minimizes the effect of extremal eigenvalues contained in $\sigma_0(\Psi_{z})$:

    \begin{equation}
    \mathcal{R}_l(\lambda) = \prod_{i=1}^{l} | 1-\frac{\lambda}{\mu_i}|.
    \label{eq: R_l}
    \end{equation}

    For the eigenvalues $[\lambda_a,\lambda_b]$ in the set $\sigma_1(\Psi_{z})$, a matrix polynomial $\mathcal{Q}_{r-l}(\lambda)$ can be defined.  It is assumed that these eigenvalues are evenly spread out.

    Since these two sets distinctly contain all of the eigenvalues of $\Psi_{z}$, then the matrix polynomials have the following relationship:

    \begin{equation}
    \mathcal{P}_r(\lambda) = \mathcal{R}_l(\lambda)\mathcal{Q}_{r-l}(\lambda).
    \end{equation}

    Substituting into \eqref{eq:constr_upper_bound} and finding the matrix polynomial that minimizes the right side:
    \begin{align}
    \| \vect{e}_r\|_{\Psi_{z}}^2 
    &\leq \min\limits_{\substack{\mathcal{R}_l,\mathcal{Q}_{r-l}}} \max\limits_{\substack{\lambda}}[\mathcal{R}_l(\lambda)]^2[\mathcal{Q}_{r-l}(\lambda)]^2 \| \vect{e}_0\|_{\Psi_{z}}^2 \\
    &\leq \min\limits_{\substack{\mathcal{R}_l}} \max\limits_{\substack{\lambda}}[\mathcal{R}_l(\lambda)]^2\min\limits_{\substack{\mathcal{Q}_{r-l}}} \max\limits_{\substack{\lambda \in \sigma_1(\Psi_{z})}}[\mathcal{Q}_{r-l}(\lambda)]^2 \| \vect{e}_0\|_{\Psi_{z}}^2.
    \end{align}

    Substituting in the chosen matrix polynomial $\mathcal{R}_l(\lambda)$ from \eqref{eq: R_l} and choosing a Chebyshev polynomial form for $\mathcal{Q}_{r-l}(\lambda)$ due to the equal spacing assumption, we obtain

    \begin{equation}
    \| \vect{e}_r\|_{\Psi_{z}}^2 \leq [\max\limits_{\substack{\lambda \in \sigma_1(\Psi_{z})}} \prod_{i=1}^{l} | 1-\frac{\lambda}{\mu_i}|]^2[(2\frac{\sqrt{\frac{\lambda_b}{\lambda_a}} - 1}{\sqrt{\frac{\lambda_b}{\lambda_a}} +1})^{r-l}]]^2 \| \vect{e}_0\|_{\Psi_{z}}^2 \text{, and}\\
    \end{equation}

    \begin{equation}
    \| \vect{e}_r\|_{\Psi_{z}}^2 \leq [C]^2
    [(2\frac{\sqrt{\operatorname{cond}(\Psi_{z})_{eff}} - 1}{\sqrt{\operatorname{cond}(\Psi_{z})_{eff}} +1})^{r-l}]^2 \| \vect{e}_0\|_{\Psi_{z}}^2 .\\
    \end{equation}

    Taking the square root of both sides gives an upper limit error bound with respect to the energy norm:

    \begin{equation}
    ||\vect{e}_r||_{\Psi_{z}} \leq 2C(\frac{\sqrt{\operatorname{cond}(\Psi_{z})_{eff}} - 1}{\sqrt{\operatorname{cond}(\Psi_{z})_{eff}} +1})^{r-l} \| \vect{e}_0\|_{\Psi_{z}}.
    \end{equation}
\end{proof}

{Thus, we have shown that we obtain a sharper error bound with respect to the covariance matrix $\Psi_z$.  In the context of the CGD-EnKF data assimilation algorithm, there are some specific underlying assumptions that are required for either of the discussed bounds to hold.  As noted in \eqref{eq:Psi_zz}, the symmetric positive definite $\Psi_z$ consists of the addition of the predicted observation ensemble covariance $\frac{1}{N-1} \sum_{i=1}^{N}(\vect{z}_{i} - \vect{\bar{z}})(\vect{z}_{i} - \vect{\bar{z}})^T$ and the observed observation covariance $R$.  Therefore, for $\Psi_z$ to be symmetric positive definite, a sufficient but not necessary assumption is that $R$ must be symmetric positive definite by construction, which is a typical assumption in data assimilation already.  Apart from this the standard EnKF assumption of linear-approximately Gaussian conditions must hold for the convergence of our proposed CGD-EnKF algorithm to the standard EnKF method.  The numerical experiments to verify the above analyses are shown in Section \ref{sec:ex1}. 
}

\begin{remark}
Note that the constructed bound in Proposition \ref{prop:3} will only be sharper than the classical bound in Proposition \ref{prop:1} if either the eigenvalue gap between $\lambda_b$ and $\lambda_{b-1}$ or $\lambda_a$ and $\lambda_{a+1}$ is sizable enough.  This gap is problem-dependent, but often for $\Psi_z$, one of these gaps becomes larger as $N$ increases.
\end{remark}

\section{Equivalent matrix formulations of EnKF algorithms}
\label{sec:matrix_reform_alg}

In this section, we present matrix formulations of the previously introduced EnKF algorithms to facilitate implementation. We further focus on the regime $N\ll p = m$.

Let us denote 
\(\mathbbm{1}_N=(\underbrace{1,\ldots,1}_{N})^T\) 
and introduce the projection on the ``constant vector" (an averaging operator): 
\(Q_{\mathbbm{1}}=\frac{1}{N}\mathbbm{1}_N\mathbbm{1}_N^T\) and denote $R_N=(N-1)R$, recalling $R\in\mathbb{R}^{p\times p}$, for notation simplicity. For $Q_{\mathbbm{1}}$ we have
\begin{equation}\label{Q-1}
Q_{\mathbbm{1}}^2=Q_{\mathbbm{1}},\quad Q_{\mathbbm{1}}=Q_{\mathbbm{1}}^T, \quad
(I-Q_{\mathbbm{1}})^2=(I-Q_{\mathbbm{1}}),\quad Q_{\mathbbm{1}}(I-Q_{\mathbbm{1}})=0. 
\end{equation}

Given $Z,\widehat{Z}\in\mathbb{R}^{p\times N}$ as defined in Section \ref{sec:EnKF} with $p\in \mathbb{N}$, we can equivalently denote $\widehat Z=Z(I-Q_{\mathbbm{1}})$.

For any vector $\bm w_N\in \mathbb{R}^N$, we have that 
\begin{equation}
\bm w_N - \overline{\bm w}_N:=(I-Q_{\mathbbm{1}})\bm w_N,
\quad \mbox{where}\quad \overline{\bm w}_N=\frac{1}{N}\sum_{i=1}^N[\bm w_N]_i.
\end{equation}

Using these notations and given the EnKF analysis step as presented in Figure \ref{fig:comparison} (left), we can write the following simplified matrix formulation of the EnKF analysis step, seen in Algorithm~\ref{enkf-0}.

\begin{algorithm}[H]
\caption{Exact EnKF Analysis (Matrix Formulation)\label{enkf-0}}
\begin{algorithmic}[1]
\Function{enkf\_exact}{$X^f$}
\State     Set \(Z = H X^f\), \(\overline{\bm{z}}=\frac{1}{N}Z\mathbbm{1}\), and \(\widehat Z=Z(I-Q_{\mathbbm{1}})\); 
\State \(
X^a=X^f+X^f\widehat{Z}^T\left(\widehat{Z}\widehat{Z}^T+R_N\right)^{-1}\left(\left(\bm{y}-\overline{\bm{z}}\right)\mathbbm{1}_N^T-\widehat{Z}\right)
\)
\State \Return{$X^a$}
\EndFunction
\end{algorithmic}
\end{algorithm}

\subsection{EnKF-Reduced Matrix Formulation}

We now rewrite Algorithm~\ref{enkf-0} in another alternative matrix formulation to show a different application of CGD suitable for the specific case $N\ll p$.  We will consider some of the same computational efficient formulation concepts as seen in Ensemble Square Root Filters \cite{Tippett2003, Mandel2006}.

Let $R=LL^T$ where $L \in\mathbb{R}^{p\times p}$ is lower triangular. If $R$ is diagonal, this decomposition is trivial. If $R$ is not diagonal, we assume this decomposition is with a sparse $L$. This leads to the following modification in the last step of Algorithm~\ref{enkf-0}:
\begin{equation}\label{LL-0}
\begin{aligned}
X^a &= X^f\widehat{Z}^TL^{-T}\left(I+L^{-1}\widehat{Z}\widehat{Z}^TL^{-T}\right)^{-1}L^{-1}Y\\
&~~~~~~~~~~~~~~~~~~+X^f\left( I - \widehat{Z}^TL^{-T}\left(I+L^{-1}\widehat{Z}\widehat{Z}^TL^{-T}\right)^{-1}L^{-1}Z\right).
\end{aligned}
\end{equation}

To simplify ~\eqref{LL-0}, we denote \(P=L^{-1}\widehat{Z}\), $\overline{\bm{z}}=\frac{1}{N}Z\mathbbm{1}$, and 
$\widehat{Y}=(\bm{y}-\overline{\bm{z}})\mathbbm{1}^T$  and we have
\begin{equation}\label{LL}
\begin{aligned}
X^a &= 
X^f+ X^f\widehat{Z}^TL^{-T}\left(I+L^{-1}\widehat{Z}\widehat{Z}^TL^{-T}\right)^{-1}L^{-1}(\widehat{Y}-\widehat{Z})\\
&=X^fP^T\left(I+PP^T\right)^{-1}L^{-1}\widehat{Y}+
X^f\left(\underbrace{I - P^T\left(I+PP^T\right)^{-1}P}_{(I+P^TP)^{-1}}\right)\\
&=X^fP^T\left(I+PP^T\right)^{-1}L^{-1}\widehat{Y}+
X^f\left(I+P^TP\right)^{-1}.
\end{aligned}
\end{equation}
(Note that this $P$ is different from $P^a$ or $P^f$ as utilized in earlier sections.) 

We can then utilize the Duncan formula \cite{1944DuncanW-aa}, also referred to as the Sherman-Morrison-Woodbury Formula \cite{1950ShermanJ_MorrisonW-aa, 1950WoodburyM-aa},  to achieve \((I+PP^T)^{-1}=I-P(I+P^TP)^{-1}P^T\)
and to get everything in terms of \((I+P^TP)\in \mathbb{R}^{N\times N}\). Thus, we obtain
\begin{equation}\label{LL-2}
\begin{aligned}
X^fP^T
&\;
(I-P(I+P^TP)^{-1}P^T)L^{-1}\widehat{Y}\\
&=
X^f\left(P^T-(P^TP+I-I)(I+P^TP)^{-1}P^T\right)L^{-1}\widehat{Y}\\
&=X^f(I+P^TP)^{-1}P^T L^{-1}\widehat{Y}.
\end{aligned}
\end{equation}
This gives the following second matrix formulation of the EnKF analysis step in Algorithm \ref{enkf-reduced}.  Note that this reformulation is called EnKF-Reduced as the computational complexity of the data assimilation step is dominated by the ensemble size $N$ rather than the observation dimension $p$.  Thus, we recover some additional efficiency gains with EnKF-Reduced in the special case where $N \ll p$.

\begin{algorithm}[H]
\caption{EnKF-Reduced Analysis (Matrix Formulation 2) \label{enkf-reduced}}
\begin{algorithmic}[1]
\Function{EnKF\_Reduced}{$X^f$}
\State Set $Z = H X^f$, $R_N=L L^T$, $P=L^{-1}\widehat{Z}\in\mathbb{R}^{p\times N}$, and 
$\overline{\bm{z}}=\frac{1}{N}Z\mathbbm{1}$; 
\State
\(
X^a = X^f\left(I+P^TP\right)^{-1}\left[I + P^TL^{-1}\left(\bm{y}-\overline{\bm{z}}\right)\mathbbm{1}_N^T
\right]
\)
\State \Return{$X^a$}
\EndFunction
\end{algorithmic}
\end{algorithm}

\begin{remark}
    If $R$ is diagonal or if  the fill-in in $L$ from $R_N=LL^T$ is $\mathcal{O}(p)$ then the action of 
$L^{-1}$ on each column of $Z$ require no more than $\mathcal{O}(p)$ multiplications. 
\end{remark}

\subsection{CGD-EnKF-Reduced Matrix Formulation}

With the above formulation in Algorithm \ref{enkf-reduced}, the analysis step is combined into a compressed one step update.  This update allows for the classical application of CGD on the entire analysis step, as has been seen in previous works applying CGD to EnKF algorithms \cite{Bardsley2013,Bardsley2013-2}.  This will be referred to as CGD-EnKF-Reduced.  We have the following result regarding the CGD algorithm for computing the action of 
low rank perturbations of the identity, such as $\left(I+L^{-1}\widehat{Z}\widehat{Z}^TL^{-T}\right)^{-1}$.
\begin{lemma}\label{l-1} If no round-off error and $N < p$, then solving the linear system 
\[
\left(I+L^{-1} \widehat{Z} \widehat{Z}^T L^{-T}\right)\bm{v}=\bm{g}
\] 
requires at most $(N+1)$ CGD steps to find the solution $\bm{v}$.
\end{lemma}

\begin{proof} Let 
$P=L^{-1} \widehat{Z}$ with columns $\left\{\bm{p}_i\right\}_{i=1}^N$ and 
$\mathcal{V}=\operatorname{span}\left\{\bm{p}_i\right\}_{i=1}^N$ and let 
$\bm \varphi\in \mathcal{V}^{\perp}$. Then $\bm \varphi$ is an eigenvector of 
$\left(I+PP^T\right)$ corresponding to the eigenvalue $1$. We have
\(\dim \mathcal{V}\le N\), and \(\dim \mathcal{V}^\perp \ge (p-N)\). Therefore, $(I+PP^T)$ has at least $(p-N)$ eigenvalues equal to $1$ and at most $N$ eigenvalues different from $1$. 
Since the eigenvalues are the roots of the minimal polynomial for $\Theta_z=L^{-1} \widehat{Z} \widehat{Z}^T L^{-T}$, it follows that 
the minimal polynomial is of degree at most $(N+1)$. On the other hand, the conjugate gradient method finds 
$(I+P^TP)^{-1} \bm{g}$ in number of iterations less than or equal  to the degree of the minimal polynomial.  
\end{proof}

\begin{remark}
    When $R$ is a diagonal matrix (Recall that $R_N=(N-1)R$), we arrive at
\begin{equation}
    I+P^TP = I+\widehat{Z}^TR_N^{-1}\widehat{Z}.
\end{equation}
Notice that by definition 
\( (I+P^TP) = \left(I+\widehat{Z}^TL^{-T}L^{-1}\widehat{Z}\right)\) 
is an $N\times N$ matrix, and for $N\ll p$, this inverse can be computed directly without using any conjugate gradient.  We may, however, use conjugate gradient with the right-hand side set to be the columns of the identity to obtain an approximation: 
\[
(I+P^TP)^{-1}\approx q_{\text{CGD}}(P^TP),
\]
where $q_{CGD}(x)$ is a polynomial which approximates 
$\frac1x$ \cite{Kraus2012}. Such a polynomial can be constructed using CGD to obtain an approximation to \((I+P^TP)^{-1}\). 
In all cases, the degree of this polynomial is equal to the number of (CGD) steps we take. 
\end{remark}

\begin{remark}
    If $N$ is held fixed and we chase asymptotics in $p\to\infty$, then the CGD allows us to propagate the exact EnKF, i.e. perform the map $X^a\leftarrow X^f$ 
in $\mathcal{O}(p)$ steps.
\end{remark}

\begin{remark}
    If we use parallelized $L^{-1} Z$, then this can be a very fast algorithm given a computationally inexpensive Cholesky factorization.
\end{remark}

\begin{remark}
    While this application of CGD is beneficial for the case $N \ll p$, in the more general case where $N$ is not bounded, it can be shown to be more efficient to utilize CGD-EnKF as presented in Figure \ref{fig:comparison_CG}.
\end{remark}

\begin{remark}
    We can derive sEnKF from EnKF under this matrix formulation.  This derivation is provided in Appendix \href{sec:app}{A}.
\end{remark}

\subsection{Parallel Implementation of CGD-EnKF-Reduced} 
\label{sec:parallel_reduced} 

As stated, CGD-EnKF-Reduced is also parallelizable, as well as certain operations under EnKF-Reduced.  A disucssion of one parallel implementation of the EnKF-Reduced algorithm is provided in \cite{Houtekamer2014}.  We will adopt the same row-block distribution of $p$ across $n_p$ processors as in Section~\ref{sec:parallel_cgd} and demonstrate and discuss the computational cost of parallelized versions of these algorithms.  Specifically, we will separate the cost comparison into two cases: general $R$ structure and tridiagonal Toeplitz $R$.           

  \subsubsection{General complexity}

  The reduced formulations require the Cholesky factorization $R_N = LL^T$ and the computation of $P = L^{-1}\widehat{Z}$ via $N$ forward substitutions.  Let $c_R$ denote the cost of the Cholesky factorization and $c_L$ denote the cost of a single forward solve. 
  The total cost of obtaining $P$ is
  $c_R + N c_L$ where $c_R$ and $Nc_L$  depend on the structure of
  $R$.  For example, general dense $R$ yields $c_R = \mathcal{O}(p^3)$ and $c_L = \mathcal{O}(p^2)$ making the factorization just as expensive as the matrix inversion required by EnKF; however, diagonal or tridiagonal $R$ yield $c_R = \mathcal{O}(p)$ and $c_L = \mathcal{O}(p)$ allowing for a much cheaper factorization.  Since the $N$ columns of $\widehat{Z}$ are independent, the forward
  solves can always be parallelized by distributing columns across
  processors, giving $Nc_L/n_p$ per processor
  regardless of $L$'s structure.  Apart from the Cholesky factorization and formulation of $P$, the remainder of the costs are as typically given for each respective method.  Table~\ref{tab:complexity_general} summarizes the costs given $N \ll p$.  It is worth noting that the three methods are comparable in order of computational cost if $c_R$ and $c_L$ are cheap to compute.  In addition, if CGD-EnKF-Reduced is parallelized and given small $r$, it will be computationally faster compared with EnKF-Reduced.

  \begin{table}[ht]
  \centering
  \footnotesize 
  \caption{Comparison of the analysis step of EnKF-based reduced algorithms on computational complexity for general $R$ where $N \ll p$ and $n_p$ processors available for parallel implementation capabilities.}
  \label{tab:complexity_general}
  \begin{tabular}{l ccc}
  \hline
  \textbf{Operation}
    & \textbf{EnKF-Reduced}
    & \textbf{CGD-EnKF-Red.}
    & \textbf{sEnKF} \\
  \hline
  $R_N\!=\!LL^T$
    & $c_R$
    & $c_R$
    & --- \\
  $P\!=\!L^{-1}\widehat{Z}$
    & $N \cdot c_L$
    & $N \cdot c_L$
    & --- \\
  $N\!\times\!N$ matrix form
    & $P^TP$: $\mathcal{O}(pN^2)$
    & $P^TP$: $\mathcal{O}(pN^2)$
    & --- \\
  $N\!\times\!N$ matrix solve
    &  Direct Inverse: $\mathcal{O}(N^3)$
    & Block CGD: $\mathcal{O}(rN^3)$
    & --- \\
  Ensemble update
    & $\mathcal{O}(pN^2)$
    & $\mathcal{O}(pN^2)$
    & $\mathcal{O}(pN^2)$ \\
  \hline
  \textbf{Total}
    & $c_R\!+\!Nc_L\!+\!\mathcal{O}(pN^2\!+\!N^3)$
    & $c_R\!+\!Nc_L\!+\!\mathcal{O}(pN^2\!+\!rN^3)$
    & $\mathcal{O}(pN^2)$ \\
  \textbf{Parallel ($n_p$ proc.)}
    & $c_R\!+\!\frac{Nc_L}{n_p}\!+\!\mathcal{O}(\frac{pN^2}{n_p}\!+\!N^3)$
    & $c_R\!+\!\frac{Nc_L}{n_p}\!+\!\mathcal{O}(\frac{pN^2}{n_p}\!+\!\frac{rN^3}{n_p})$
    & $\mathcal{O}(\frac{pN^2}{n_p})$ \\
  \hline

  \end{tabular}
  \end{table}

  \subsubsection{Parallel complexity under tridiagonal Toeplitz $R$}
    We will now discuss one such matrix structure of $R$ that results in cheap $c_R$ and $c_L$.  In the numerical experiments under Section~\ref{sec:ex2}, the observation noise covariance is a tridiagonal Toeplitz matrix with diagonal entry $\sigma^2$ and off-diagonal entry
  $\sigma^2\rho$.  The factor $L$ is bidiagonal resulting in
  $c_L = \mathcal{O}(p)$, and $L$ here has a Toeplitz structure resulting in $c_R = \mathcal{O}(1)$. 
  Therefore, $c_R$ and $c_L$ are relatively cheap to produce and therefore, do not dominate the computational cost.  Instead the matrix inversion or inverse approximation will dominate as previously seen in EnKF and CGD-EnKF.  The resulting costs given this tridiagonal Toeplitz $R$ are given in Table~\ref{tab:complexity_impl}.

  \begin{table}[ht]
  \centering
  \footnotesize 
  \caption{Comparison of the analysis step of EnKF-based reduced algorithms on computational complexity for tridiagonal Toeplitz $R$ where $N \ll p$ and $n_p$ processors available for parallel implementation capabilities.}
  \label{tab:complexity_impl}
  \begin{tabular}{l ccc}
  \hline
  \textbf{Operation}
    & \textbf{EnKF-Reduced}
    & \textbf{CGD-EnKF-Red.}
    & \textbf{sEnKF} \\
  \hline
  $P\!=\!L^{-1}\widehat{Z}$
    & $\mathcal{O}(pN/n_p)$
    & $\mathcal{O}(pN/n_p)$
    & --- \\
  $N\!\times\!N$ matrix
    & $\mathcal{O}(pN^2/n_p)$
    & $\mathcal{O}(pN^2/n_p)$
    & --- \\
  $N\!\times\!N$ solve
    & $\mathcal{O}(N^3)$
    & $\mathcal{O}(rN^3/n_p)$
    & --- \\
  Ensemble update
    & $\mathcal{O}(pN^2/n_p)$
    & $\mathcal{O}(pN^2/n_p)$
    & $\mathcal{O}(pN^2/n_p)$ \\
  \hline
  \textbf{Total (per proc.)}
    & $\mathcal{O}(pN^2/n_p\!+\!N^3)$
    & $\mathcal{O}(pN^2/n_p\!+\!rN^3/n_p)$
    & $\mathcal{O}(pN^2/n_p)$ \\
 
  \hline
  \end{tabular}
  \end{table}

\section{Numerical Experiments}
\label{sec:experiments}

In this section, we provide several numerical examples to validate the proposed algorithms and illustrate their accuracy and computational efficiency.  Section \ref{sec:ex1} provides an example providing a state estimation of a constant random valued vector to validate the method's accuracy against known EnKF models and the analysis provided in Section \ref{sec:errorAnalysis}.  Section \ref{sec:ex2} showcases the computational efficiency of CGD-EnKF and CGD-EnKF-Reduced for highly dimensional systems through the benchmark Lorenz-96 model.  The final provided example in Section \ref{sec:ex3} demonstrates the capabilities of CGD-EnKF by considering a Darcy flow model PDE.

The implementation for the experiments in Sections \ref{sec:ex1} and \ref{sec:ex3} are based in Python and were run with CPU on Apple M3; whereas, {the experiments in Section~\ref{sec:ex2}
were implemented in C with MPI and executed on a single Amazon EC2
c5.24xlarge instance equipped with 96 virtual CPUs and 192 GiB of
memory with 50 CPUs utilized.}

\subsection{Example 1. Simple Test Case: Random Valued Vector}
\label{sec:ex1}

The objective in this first example is to validate the solution convergence and verify our analytical upper bound discussed for our proposed CGD-EnKF in the previous section. 
Here, we consider the approximation of a random valued-vector up to the level of noise.

We define the true constant state of the random value vector as ${\bf x}^t \in \mathbb{R}^m$, where ${\bf x}^t$ is generated from a standard Gaussian centered at ${\bf 0}$.  We set an ensemble size of $N=50$, and the initial ensemble $X_0^a$ is sampled from the same Gaussian distribution as ${\bf x}^t$.  We let $F_k=H_k=I \in \mathbb{R}^{m \times m}$ where $m=p$ for all $k$.  {Preconditioning for CGD in CGD-EnKF is not required for this example as cond($\Psi_z$) is reasonably small in all test cases.}

\subsubsection{Example 1a. Error Validation of CGD-EnKF}

First, we validate the error convergence of the CGD-EnKF method as $k \rightarrow \infty$ and demonstrate that the resulting state estimations coincide with those of classical EnKF.  We set $m=p=100$ and let $Q_k=0.5I 
\in \mathbb{R}^{m \times m}$, and the  observations are produced by adding Gaussian noise
\begin{equation}
    {\bf y}_k = H_{k}\vect{x}^{\text{t}} + \vect{v}_k, \qquad \vect{v}_k \sim \mathcal{N}(\vect{0},R_k),
\end{equation}
where $R_k$ is chosen as a dense  form
    \begin{equation}
    R_k = \bigl(\rho + \sigma^2 \bigr)I + \rho S \in \mathbb{R}^{m \times m},
    \end{equation}
    with $\sigma^2 = 0.1$, $\rho = 0.5$.
Here, $S=\mathbf{1}\in\mathbb{R}^{m \times m}$, a matrix of ones.
The $L_2$ error over 2000 filtering cycles is tracked under the conditions described above by computing:
$\|{\bf x}^t - \bar{\bf x}_k^a\|_2$
at each iteration step $k=1,...,2000$.

Figure \ref{fig:CGD_EnKF_Error_Compare} shows decreasing error over the $2000$ iterations and indicates convergence to a minimum error related to the noise of the data assimilation process. 
Moreover, we observe that the CGD-EnKF error converges to the EnKF error and tracks closely to the sEnKF error as $k \rightarrow \infty$. 
We note that the error for sEnKF slightly differs due to the significant off-diagonal contributions $R_k$. As it was discussed before, sEnKF only considers the diagonal and neglects to consider the cross-variance that the off-diagonal elements provide.  For even simple cases such as this random-valued vector approximation, this unaccounted for correlations can cause greater error during portions of the EnKF cycles and an unstable error trajectory as sEnKF attempts to learn the correlations without knowing the full $R_k$.  Yet, our proposed method delivers the same order of error as EnKF and sEnKF.

To further validate our method, we consider the $L_2$ error of CGD-EnKF at the 500th time step, 
$ \|{\bf x}^t - \bar{\bf x}_{500}^a\|_2$,   
and vary the ensemble size $N$.  Figure \ref{fig:CGD_EnKF_Conv_N} confirms that as $N$ increases, the mean state estimation becomes closer to the true state and converges to some level of noise.  Note that this trend holds for any chosen $k>0$ but we have shown the case where $k=500$. This classical test confirms that adding ensemble members leads to more accurate state estimations which is a classical trait of EnKF algorithms.

\begin{figure}[t]
\centering
\begin{subfigure}[t]{0.45\textwidth}
\includegraphics[width=\textwidth]{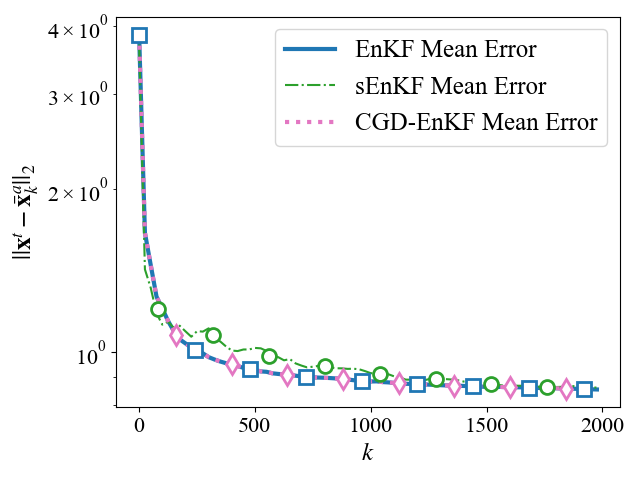}
\caption{}
\label{fig:CGD_EnKF_Error_Compare}
\end{subfigure}
\begin{subfigure}[t]{0.5\textwidth}
\includegraphics[width=\textwidth]{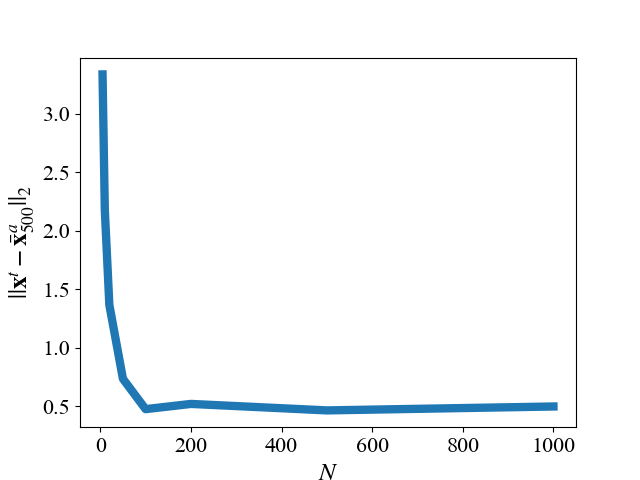}
\caption{}
\label{fig:CGD_EnKF_Conv_N}
\end{subfigure}  
\caption{Example 1a. (a) $\|{\bf x}^t - \bar{\bf{x}}_k^a\|_2$ comparing Kalman-based Ensemble Methods given $m=p=100, N=50$ (b) $L_2$ Error at $k=500$ varying $N$ where $m=p=100$ }
\end{figure}

\subsubsection{Example 1b. Validating CGD Error Bounds and Convergence}

Utilizing the setup given in Section \ref{sec:ex1}, but with $Q_k=0.1I$ and $R_k=I$, we demonstrate that CGD-EnKF observes the upper error bounds given in Section \ref{sec:errorAnalysis}.  Additionally, we verify that the state estimates of CGD-EnKF converge to the EnKF estimates under our convergence proposition of Section \ref{sec:errorAnalysis}.

We compute the CGD Error as earlier defined as 
\begin{equation}
    \vect{e}_r = K^T - (K^T)_r = \Psi_{z}^{-1}\Psi_{x}^T - (\Psi_{z}^{-1})_r\Psi_{x}^T
    \label{eq:e_r_def}
\end{equation}
for each CGD iteration $r$ at the same first step of CGD-EnKF ($k=1$) for the dimensions $m=p=[10,20, 200, 4000]$.  For each dimension, we evaluate the average error over ensemble members at each CGD iteration under the energy norm $||\cdot||_{\Psi_{z}}$, and we set the error tolerance to be $\epsilon = 10^{-12}$ for CGD.  Results for testing the classical error bound in Proposition \ref{prop:1} are shown in Figure \ref{fig:cg_error}. 

The CGD error decreases as the number of CGD iterations increases for all dimensions tested.  This indicates that the solution from CGD-EnKF converges to the true solution as proposed.  Additionally, the CGD error is bounded above by the classical Chebyshev upper error bound.

\begin{figure}[htbp]
    \centering
    \begin{subfigure}[t]{0.45\textwidth}
        \centering
        \includegraphics[width=\linewidth]{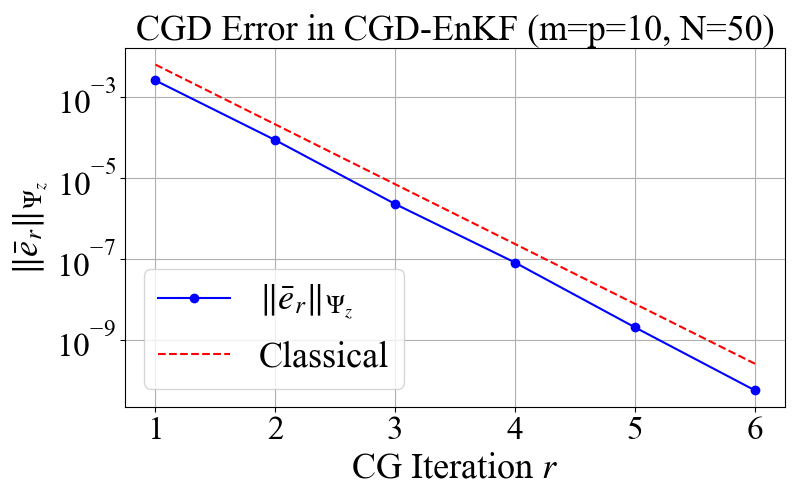}
        \caption{CGD Classical Bound ($m=p=10$)}
    \end{subfigure}
    \hfill
    \begin{subfigure}[t]{0.45\textwidth}
        \centering
        \includegraphics[width=\linewidth]{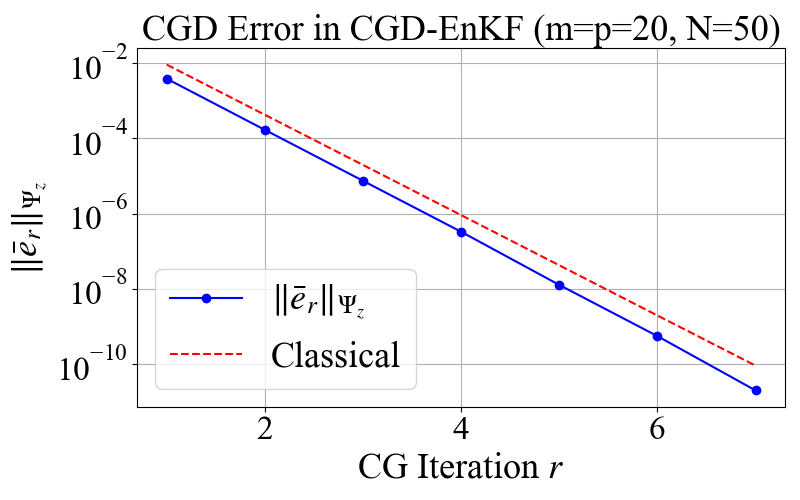}
        \\  
        \caption{CGD Classical Bound ($m=p=20$)}
    \end{subfigure}

    \vspace{0.5cm}

    \begin{subfigure}[t]{0.45\textwidth}
        \centering
        \includegraphics[width=\linewidth]{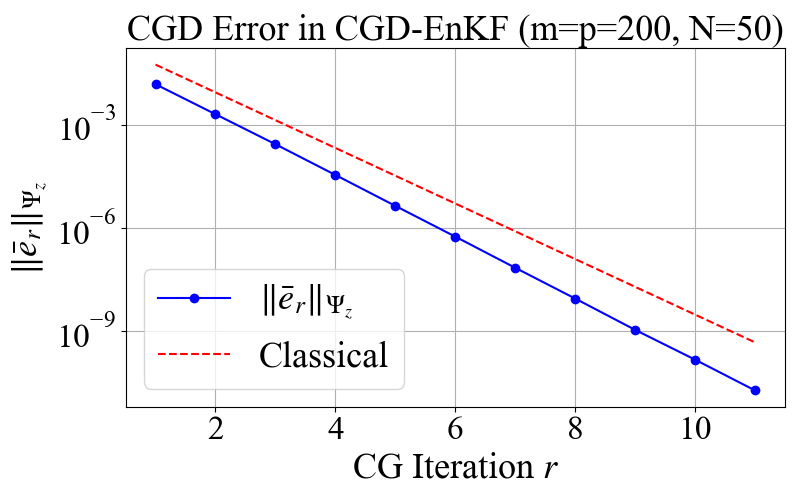}
        \caption{CGD Classical Bound ($m=p=200$)}
    \end{subfigure}
    \hfill
    \begin{subfigure}[t]{0.45\textwidth}
        \centering
        \includegraphics[width=\linewidth]{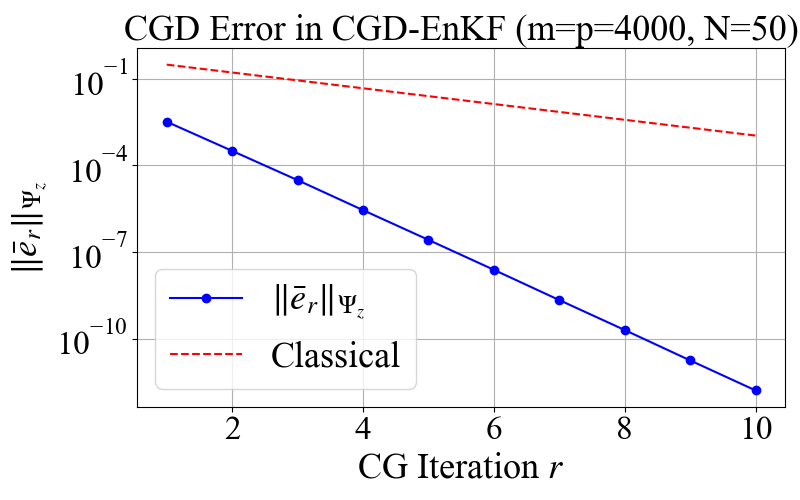}
        \caption{CGD Classical Bound ($m=p=4000$)}
    \end{subfigure}
    \caption{Example 1b.  CGD error and classical error bound \cite{Hestenes1952} for random-valued vector model with different state dimensions}
    \label{fig:cg_error}
\end{figure}

Note that the classical bound is not as sharp when we increase the dimensionality of the problem.  This occurs as the classical error bound assumes equidistant eigenvalues; however, as dimensionality increases, the eigenvalues begin to cluster, and outliers are observed.  Work on error bounds for approximating eigenvalues has been studied and shown deep connections to the error bounds for CGD and other subspace methods \cite{Haas2025}.  The emerging eigenvalues for this particular example are illustrated with $m=p=4000$ in Figure \ref{fig:better_bound} (a).  Note that two clusters of eigenvalues emerge.  It should also be noted that the majority of the eigenvalues are non-unique and form a cluster at $\mu=1$.  The remaining cluster contains eigenvalues such that $\lambda_i>7$.

We aim to show that the error bound from Proposition \ref{prop:3} also bounds the CGD error for the higher dimensional case, but is sharper than the classical error bound.  We will also consider several alternative proposed heuristic upper bounds:
\begin{itemize}
    \item Initial Design Bound: \begin{align*}\| \vect{e}_r\|_{\Psi_{z}} \leq 2(\frac{\sqrt{\widetilde{\operatorname{cond}}(\Psi_{z})_{\text{eff,2}}} - 1}{\sqrt{\widetilde{\operatorname{cond}}(\Psi_{z})}_{\text{eff,2}} +1})^{r} \| \vect{e}_0\|_{\Psi_{z}}, \widetilde{\operatorname{cond}}(\Psi_{z})_{\text{eff,2}}=\frac{\lambda_b}{\left(\frac{\lambda_a - \lambda_{a+1}}{2}\right)} \end{align*}

    \item Jia Bound - \cite{Jia2001}: 
    \begin{align*}
    \| \vect{e}_r\|_{\Psi_{z}} \leq \frac{4}{\eta}(\frac{\sqrt{\operatorname{cond}(\Psi_{z})} - 1}{\sqrt{\operatorname{cond}(\Psi_{z})} +1})^{2r} \| \vect{e}_0\|_{\Psi_{z}}^2, \quad \text{if } \quad \eta > 1 \\
    \| \vect{e}_r\|_{\Psi_{z}} \leq 2(\frac{\sqrt{\operatorname{cond}(\Psi_{z})} - 1}{\sqrt{\operatorname{cond}(\Psi_{z})} +1})^{r} \| \vect{e}_0\|_{\Psi_{z}}, \quad \text{otherwise }\end{align*}

    \item Li Bound - \cite{Li2005}:
    \begin{align*}
    \| \vect{e}_r\|_{\Psi_{z}} \leq \frac{2 ||E||_2^2}{\eta+\sqrt{\eta^2+4||E||_2^2}}, 
    E = 2(\frac{\sqrt{\operatorname{cond}(\Psi_{z})} - 1}{\sqrt{\operatorname{cond}(\Psi_{z})} +1})^{r} \| \vect{e}_0\|_{\Psi_{z}}, \quad \text{if } \quad \eta > 1 \\
    \| \vect{e}_r\|_{\Psi_{z}} \leq 2(\frac{\sqrt{\operatorname{cond}(\Psi_{z})} - 1}{\sqrt{\operatorname{cond}(\Psi_{z})} +1})^{r} \| \vect{e}_0\|_{\Psi_{z}}, \quad \text{otherwise }
\end{align*}
\end{itemize}
where $\eta = \max{\{\lambda_{a}-\lambda_{a+1}\}}$ for $i=1,...n$.  The Initial Design Bound is a construction of our design that works well only when the eigenvalue clusters for small $\eta$.  This was a precursor to the bound found in Proposition \ref{prop:3}.  The Jia Bound and the Li Bound are defined based on \cite{Jia2001} and \cite{Li2005}, respectively.  These bounds have a slightly steeper slope than the classical, yet still are subject to the widening gap with the actual error for later iterations of CGD.  While not shown in this example, the Li Bound can maintain the classical error bound for the linear-$\operatorname{cond}(\Psi_{z})$ error bound for early iterations of CGD and then can slowly converge to the trajectory for the Jia Bound for later iterations if the two bounds intersect.  Note that these are still heuristic bounds and are not rigorously proven for all cases of eigenvalue distributions.

We now test the error bound from Proposition \ref{prop:3} as well as the heuristic error bounds on the higher dimensional case of $m=p=4000$.  Results are shown in Figure \ref{fig:better_bound} (b).  Note that the trends described by the heuristic bounds are observed but the proposed bound from Proposition \ref{prop:3} is the sharpest.  {Specifically, say that the error tolerance was set to be $\epsilon=10^{-7}$ for CGD.  In practice this is obtained for this example by $r=6$ where the proposed tighter upper bound guarantees the error tolerance is reached by $r=9$ and the classical bound indicates by an $r>10$.  Practically, this shows that under clustered eigenvalue distributions, the proposed bound can better estimate the number of CGD iterations required for the CGD approximation of $K^T$ under CGD-EnKF to converge to the exact $K^T$ under EnKF up to an assigned error tolerance.}

\begin{figure}[htbp]
\centering
\begin{subfigure}[t]{0.45\textwidth}
        \centering
        \includegraphics[width=\linewidth]{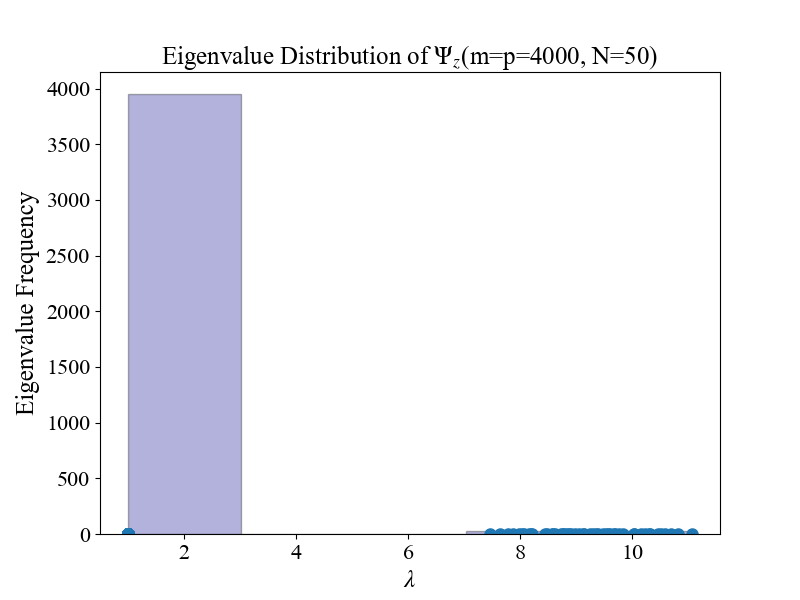}
        \caption{Eigenvalue Distribution of $\Psi_{z}$}
\end{subfigure}
\hfill
\begin{subfigure}[t]{0.5\textwidth}
        \centering        
        \includegraphics[width=\linewidth]{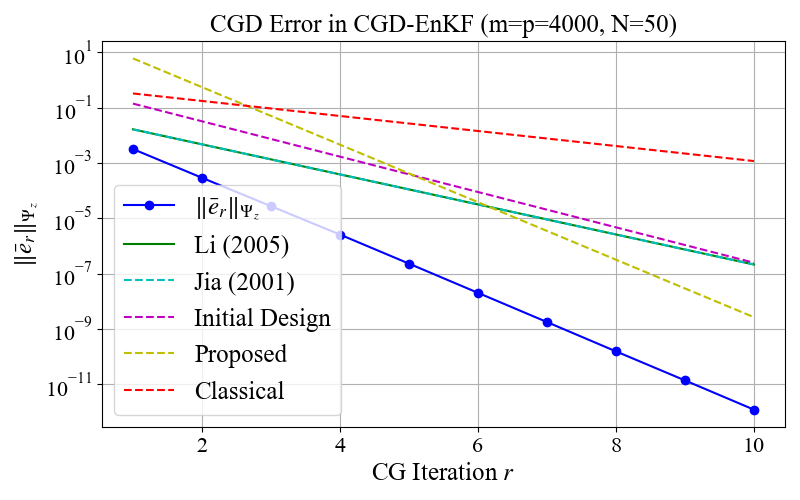}
        \\  
        \caption{$m=p=4000$ with different bounds}
\end{subfigure}
\caption{Example 1b.  Comparisons of CGD error and error bounds}
\label{fig:better_bound}
\end{figure}

\subsection{Example 2. Benchmark test case: Lorenz--96 model}

\label{sec:ex2}
In this second numerical example, we highlight the computational efficiency of CGD-EnKF for general $N$ and CGD-EnKF-Reduced for $N \ll p$ compared with EnKF and sEnKF, in particular, for highly dimensional dynamical systems, under the benchmark Lorenz--96 model. Additionally, we will compare the accuracy of the methods to examine the trade-off between computational time and state estimation error. 

{
Examples~2a and~2b apply CGD to two different size
linear systems. Example~2a uses the full observation-space
formulation whereas Example~2b uses the reduced ensemble-space formulation.
In the reduced formulation, the dimension of the CG system and
the number of right-hand sides are determined by the ensemble size
$N$, rather than by the observation dimension $p$. Consequently,
even when $p$ is extremely large, CGD is applied only to the much
smaller ensemble-space system when $N\ll p$. The observation
dimension primarily affects the formation of $P^TP$ and the
ensemble update.
}

We define the Lorenz--96 dynamical system of dimension \(m\) as

\begin{equation}
\frac{d x_{(i)}}{d t}
= \left( x_{(i+1)} - x_{(i-2)} \right) x_{(i-1)} - x_{(i)} + 8,
\qquad i = 1, \ldots, m,
\label{eq:l96}
\end{equation}
with cyclic indices \(x_{i+m} = x_i\) and constant forcing of $8$.  Note that

\begin{equation}
    {\bf x} = \begin{bmatrix}
        x_{(1)} & x_{(2)} & ... & x_{(m)}
    \end{bmatrix}^T
\end{equation}

We define a reference true state trajectory \({\bf x}_{k}^{t} \in \mathbb{R}^m\) for each assimilation step by integrating \eqref{eq:l96} with an initial condition given as
\begin{equation}
{\bf x}_{0}^{t} \sim \mathcal{N}(0, I_m).
\end{equation}

We chose 20 integration steps in order to reach a statistically stationary regime and we set $k=0,1$, indicating only one data assimilation step is carried out. The initial ensemble \(X_0^{a}\in \mathbb{R}^{m\times N}\) is constructed by perturbing the true state with small Gaussian noise:

\begin{equation}
X_0^{a} = 
\left[\, {\bf x}_{0}^{t}, \ldots, {\bf x}_{0}^{t} \,\right]
+ 0.1\, \mathcal{N}(0, I_m).
\label{eq:init_ens}
\end{equation}

The forecast step is given as 
\begin{equation}
{\bf x}_{k+1}^{f} = F_k  {\bf x}_{k}^{a} + {\bf w}_k,
\label{eq:forecast_model}
\end{equation}
where $F_k$ is obtained by discretizing the continuous Lorenz--96 system in \eqref{eq:l96} using a fourth–order Runge--Kutta integrator with time step
$\Delta t = 0.01$ and \({\bf w}_k \sim \mathcal{N}(0, Q_k)\) represents model noise with covariance \(Q_k = I_m\).

At each assimilation step $k$, an observation vector ${\bf y}_k \in \mathbb{R}^p$
is generated from the true state according to
\begin{equation}
{\bf y}_k = H\, {\bf x}_k^{\mathrm{t}} + {\bf v}_k,
\qquad
{\bf v}_k \sim \mathcal{N}(0, R_k),
\label{eq:obs_model}
\end{equation}
where $m=p$, \(H = I \in \mathbb{R}^{m\times m}\), and ${\bf v}_k \sim \mathcal{N}(0, R_k)$ represents the observation noise with tridiagonal covariance matrix with entries given by

\begin{equation}
R_{ii} = \sigma^2, \qquad
R_{i,i\pm1} = \sigma^2\rho,
\end{equation}
where $\sigma > 0$ controls the noise magnitude and $\rho \in [0,0.5)$ controls the nearest-neighbor correlation.  Notably this matrix admits an efficient Cholesky factorization due to its tridiagonal structure.  For this experiment we take $\sigma=0.3$ and $\rho = 0.2$.
{
The primary purpose of this work is to demonstrate that the linear
systems arising in the EnKF analysis can be treated using an
iterative method and, under the reduced formulation, can be
transformed into a substantially lower-dimensional system. We
therefore use the standard unpreconditioned CGD method throughout
the numerical experiments in order to illustrate the effect of the
proposed reformulation without introducing additional
solver-specific techniques.}

\subsubsection{Example 2a. Computational Efficiency of CGD-EnKF}
First, we evaluate the computational time of CGD-EnKF
under the Lorenz--96 model setup and compare it with the sEnKF baseline. In this experiment, we vary the state and observation
dimension simultaneously as
\[
m = p \in \{10,100, 1{,}000, 10{,}000, 100{,}000\}
\]
and the ensemble size is fixed at $N=1000$.

To measure computational efficiency, we record the total wall-clock
time for each analysis method (CGD-EnKF and sEnKF), which
captures their computational scalability with dimension $p$.
All reported wall-clock times correspond to MPI-parallel
implementations .

{
In the parallel implementation, the observation dimension is
distributed across MPI ranks using contiguous row blocks. The
coefficient matrix $\Psi_z$ is not formed explicitly. Instead, its
action on a CG search direction $\bm d$ is evaluated as
\[
\Psi_z\bm d
=
\frac{1}{N-1}\widehat Z
\left(\widehat Z^T\bm d\right)+R\bm d.
\]
The intermediate vector $\widehat Z^T\bm d$ is assembled through
an MPI collective reduction. Since the observation covariance
matrix is tridiagonal in this experiment, the product $R\bm d$
requires only local computation and communication between
neighboring row blocks.
}

Table~\ref{tab:cg_full_vs_serial_N1000}
presents the wall-clock runtime comparison between CGD-EnKF and the sEnKF baseline across increasing state dimensions.  CGD-EnKF is seen to take only approximately twice as long to run for a single iteration CGD iteration and further iterations adding computational cost.  Compared with other alternative EnKF, CGD-EnKF is seen to be a more scalable alternative to sEnKF while, as Example \ref{sec:ex1}, reducing the computational error that may occur for sEnKF's, diagonal $R$ requirement.

This result also demonstrates the parallelizability of CGD-EnKF allowing for high dimensional in observation space assimilation problems to be feasible in a reasonable amount of time.  This key feature further makes CGD-EnKF comparable to sEnKF whereas EnKF and other nonparallelizable data assimilation methods would take considerable computational expense or even be infeasible on such high dimensional observation spaces.

\begin{table}[ht]                                                   
  \centering                                                                     
  \small                                                                           
  \begin{tabular}{r r r r}                                                         
  \toprule                                                  
  $p$ & CGD Step $r$ & CGD-EnKF Time (s)  & sEnKF Time (s)  \\
  \midrule

  \multirow{4}{*}{10}
    & -- & --                      & $6.01\times10^{-4}$ \\
    & 1  & $2.98\times10^{-3}$     & -- \\
    & 2  & $3.74\times10^{-3}$     & -- \\
    & 3  & $4.54\times10^{-3}$     & -- \\
  \midrule

  \multirow{4}{*}{100}
    & -- & --                      & $6.10\times10^{-3}$ \\
    & 1  & $1.87\times10^{-2}$     & -- \\
    & 2  & $2.47\times10^{-2}$     & -- \\
    & 3  & $2.93\times10^{-2}$     & -- \\
  \midrule

  \multirow{4}{*}{1{,}000}
    & -- & --                      & $8.18\times10^{-2}$ \\
    & 1  & $1.85\times10^{-1}$     & -- \\
    & 2  & $2.56\times10^{-1}$     & -- \\
    & 3  & $3.25\times10^{-1}$     & -- \\
  \midrule

  \multirow{4}{*}{10{,}000}
    & -- & --                      & $2.57\times10^{0}$ \\
    & 1  & $5.05\times10^{0}$      & -- \\
    & 2  & $7.36\times10^{0}$      & -- \\
    & 3  & $9.68\times10^{0}$      & -- \\
  \midrule

  \multirow{4}{*}{100{,}000}
    & -- & --                      & $2.71\times10^{2}$ \\
    & 1  & $5.47\times10^{2}$      & -- \\
    & 2  & $7.90\times10^{2}$      & -- \\
    & 3  & $1.03\times10^{3}$      & -- \\
  \bottomrule
  \end{tabular}
  \caption{Example 2a. Wall-clock runtime comparison between CGD-EnKF and sEnKF given $N=1000$ and 50 CPU cores.}
  \label{tab:cg_full_vs_serial_N1000}
  \end{table}

\subsubsection{Example 2b. Computational Efficiency and Accuracy of CGD-EnKF-Reduced}
Maintaining the setup from Section \ref{sec:ex2}, we now consider the special regime $N \ll p$ by fixing $N=100$ and demonstrate additional computational improvements enabled by the CGD-EnKF-Reduced method that was introduced in Section~\ref{sec:matrix_reform_alg}.
We highlight the method's scalability at higher state dimensions and report the wall-clock time improvements achieved with MPI parallelism.  Additionally, we verify the accuracy benefits of CGD-EnKF-Reduced and demonstrate the tradeoff with computational time.

This example was again executed using MPI with a fixed parallel configuration of
50 CPU cores. We vary the state and observation dimension simultaneously again but as
$$p\in \{ 1{,}000, 10{,}000, 100{,}000, 1{,}000{,}000,10{,}000{,}000 \} $$.

We record the total wall-clock time (reported as the maximum over
MPI ranks) and the relative error of the analysis ensemble at $k=1$ of each method, denoted $X_1^{a,\mathrm{(method)}}$, and for each of the first five CGD iterations in CGD-EnKF-Reduced, in the Frobenius norm computed as

\begin{equation}
\label{eq:rel_frob}
\mathrm{Rel. Error}
\;:=\;
\frac{\left\| X_1^{a,\mathrm{(method)}} - X_1^{a,\mathrm{EnKF-Reduced}} \right\|_F}
     {\left\| X_1^{a,\mathrm{EnKF-Reduced}} \right\|_F}.
\end{equation}

{
Although fixed iteration counts are used here to illustrate the
accuracy--runtime trade-off, a practical implementation may terminate
CGD when the relative residual of the reduced system falls below a prescribed tolerance, or when a specified maximum
number of iterations is reached.
}

\begin{table}[p]
\centering
\caption{Example 2b. Wall-clock runtime comparison between EnKF-Reduced, CGD-EnKF-Reduced, and sEnKF given $N = 100$ and 50 CPU cores.}
\begin{tabular}{ccccc}
\toprule
$p$ & CGD steps & Method & Time (s) & Rel.\ error \\
\midrule

\multirow{7}{*}{1,000}
& -- & EnKF-Reduced     & $4.150 \times 10^{-3}$ & -- \\
& 1  & CGD-EnKF-Reduced & $3.168\times 10^{-3}$ & $5.553\times 10^{-1}$ \\
& 2  & CGD-EnKF-Reduced & $3.281\times 10^{-3}$ & $2.918\times 10^{-1}$ \\
& 3  & CGD-EnKF-Reduced & $3.427\times 10^{-3}$ & $1.290\times 10^{-1}$ \\
& 4  & CGD-EnKF-Reduced & $3.587\times 10^{-3}$ & $2.501\times 10^{-2}$ \\
& 5  & CGD-EnKF-Reduced & $3.665\times 10^{-3}$ & $1.922\times 10^{-3}$ \\
& -- & sEnKF            & $2.397\times 10^{-3}$ & $4.301\times 10^{-2}$ \\

\midrule
\multirow{7}{*}{10,000}
& -- & EnKF-Reduced     & $2.846\times 10^{-2}$ & -- \\
& 1  & CGD-EnKF-Reduced & $2.649\times 10^{-2}$ & $9.266\times 10^{-1}$ \\
& 2  & CGD-EnKF-Reduced & $2.692\times 10^{-2}$ & $8.512\times 10^{-1}$ \\
& 3  & CGD-EnKF-Reduced & $2.675\times 10^{-2}$ & $5.410\times 10^{-1}$ \\
& 4  & CGD-EnKF-Reduced & $2.708\times 10^{-2}$ & $1.688\times 10^{-2}$ \\
& 5  & CGD-EnKF-Reduced & $2.735\times 10^{-2}$ & $1.554\times 10^{-4}$ \\
& -- & sEnKF            & $2.328\times 10^{-2}$ & $5.417\times 10^{-2}$ \\

\midrule
\multirow{7}{*}{100,000}
& -- & EnKF-Reduced     & $2.808\times 10^{-1}$ & -- \\
& 1  & CGD-EnKF-Reduced & $2.692\times 10^{-1}$ & $9.920\times 10^{-1}$ \\
& 2  & CGD-EnKF-Reduced & $2.691\times 10^{-1}$ & $9.831\times 10^{-1}$ \\
& 3  & CGD-EnKF-Reduced & $2.693\times 10^{-1}$ & $5.469\times 10^{-1}$ \\
& 4  & CGD-EnKF-Reduced & $2.699\times 10^{-1}$ & $1.528\times 10^{-3}$ \\
& 5  & CGD-EnKF-Reduced & $2.697\times 10^{-1}$ & $1.295\times 10^{-6}$ \\
& -- & sEnKF            & $2.292\times 10^{-1}$ & $7.190\times 10^{-2}$ \\

\midrule
\multirow{7}{*}{1,000,000}
& -- & EnKF-Reduced     & $2.875\times 10^{0}$ & -- \\
& 1  & CGD-EnKF-Reduced & $2.845\times 10^{0}$ & $9.992\times 10^{-1}$ \\
& 2  & CGD-EnKF-Reduced & $2.847\times 10^{0}$ & $9.983\times 10^{-1}$ \\
& 3  & CGD-EnKF-Reduced & $2.844\times 10^{0}$ & $6.135\times 10^{-1}$ \\
& 4  & CGD-EnKF-Reduced & $2.846\times 10^{0}$ & $1.863\times 10^{-4}$ \\
& 5  & CGD-EnKF-Reduced & $2.855\times 10^{0}$ & $2.578\times 10^{-8}$ \\
& -- & sEnKF            & $2.448\times 10^{0}$ & $1.847\times 10^{-1}$ \\

\midrule
\multirow{7}{*}{10,000,000}
& -- & EnKF-Reduced     & $2.968\times 10^{1}$ & -- \\
& 1  & CGD-EnKF-Reduced & $2.942\times 10^{1}$ & $9.999\times 10^{-1}$ \\
& 2  & CGD-EnKF-Reduced & $2.937\times 10^{1}$ & $9.998\times 10^{-1}$ \\
& 3  & CGD-EnKF-Reduced & $2.936\times 10^{1}$ & $5.214\times 10^{-1}$ \\
& 4  & CGD-EnKF-Reduced & $2.937\times 10^{1}$ & $1.168\times 10^{-5}$ \\
& 5  & CGD-EnKF-Reduced & $2.937\times 10^{1}$ & $4.574\times 10^{-10}$ \\
& -- & sEnKF            & $2.448\times 10^{1}$ & $5.758\times 10^{-1}$ \\

\bottomrule
\end{tabular}

\label{tab:enkf_cg_scaling_updated}
\end{table}

Table~\ref{tab:enkf_cg_scaling_updated} shows that all three methods---EnKF-Reduced, CGD-EnKF-Reduced, and sEnKF---exhibit the same order of computational   
  time, since the dominant cost for each is the distributed formation of an $N \times N$ matrix from the $p \times N$ ensemble, which scales linearly in~$p$.  Among them, sEnKF is consistently the fastest, as it avoids the Cholesky factorization of $R_N$ and the $N$~bidiagonal forward substitutions required to form $P = L^{-1}\widehat{Z}$ in the reduced formulations.  EnKF-Reduced is
  slightly slower than sEnKF due to this additional setup cost, while CGD-EnKF-Reduced incurs a modest overhead proportional to the
  number of CG iterations~$r$.

{
Nevertheless, the reduced formulation decreases the dimension of
the coefficient matrix from $p\times p$ in the observation space
to $N\times N$ in the ensemble space. Since $N\ll p$, solving the
reduced linear system is no longer the dominant component of the
analysis-step runtime. This explains why the total runtime is only
weakly affected by the number of CGD iterations (and may occasionally
decrease slightly with additional iterations due to ordinary timing
variability), further highlighting the computational benefit of the
reduced formulation.
}

  The key advantage of CGD-EnKF-Reduced lies in accuracy and time tradeoff.  In the initial CGD iterations, sEnKF is computationally more efficient while producing comparable state estimates, signified by comparable relative error.  However, after only a few CGD iterations, the relative error of CGD-EnKF-Reduced decreases rapidly, surpassing that of sEnKF.  This is especially pronounced at high dimensions.  For example, given $p = 10{,}000{,}000$, CGD-EnKF-Reduced with $r = 5$ achieves a relative error on the order of $10^{-10}$, compared to sEnKF's relative error on the order of $10^{-1}$, with less than $20\%$ additional wall-clock time to produce the state estimate.  Moreover, we observe that for a fixed number of CG iterations~$r$, the
  relative error of CGD-EnKF-Reduced decreases as $p$ increases.

  Overall, these results demonstrate that CGD-EnKF-Reduced provides a compelling alternative to sEnKF in the case of sparse non-diagonal $R$: it achieves substantially higher accuracy with only a modest increase in computational cost, and its
  accuracy advantage grows with the problem dimension.

{
The present example uses a controlled Gaussian setting in order to
isolate the computational and approximation effects of replacing
the direct EnKF solve with CGD and of reducing the dimension of the
resulting linear system. The proposed reformulation changes the numerical treatment
of the EnKF analysis system, but it does not remove the underlying
statistical assumptions or limitations of EnKF.}

\subsection{Example 3. Benchmark test case: Darcy flow model}
\label{sec:ex3}
In this final example, we assess the performance of CGD-EnKF for estimating the pressure field $u(\mathbf{x})$ based on the estimated probability distribution underlying a heterogeneous permeability field $\kappa(\mathbf{x})$ in a two-dimensional steady Darcy problem:
\begin{equation}
    -\nabla\!\cdot\!\big(\kappa(\mathbf{x})\,\nabla u(\mathbf{x})\big)=\mathbf{0}
\quad \text{in } \Omega=(0,1)^2,
\label{eq:Darcy}
\end{equation}
with Dirichlet boundary conditions \(u=1\) on \(\{x_1=0\}\), \(u=0\) on \(\{x_1=1\}\), and homogeneous Neumann conditions on the remaining boundaries.  Here, we define the permeability as 
\begin{equation}
    \kappa({\bf x}) = c + \sin(2\pi x_1) \cos(2\pi x_2),
\end{equation}
where $\log c \sim \mathcal{N}(0, \sigma_\kappa^2 I)$ with truth $\log\kappa^t$ and Matérn covariance $\sigma_\kappa^2$ with correlation length \(L\).

We aim to approximate only the true pressure field $u^t$.  We sample $N=100$ ensemble members denoted as $\log \kappa_0^{(j)}$ for $j=1,...,100$.  The forward model $F_0$ is given as the discretization scheme of \eqref{eq:Darcy}.  All subsequent assimilation cycles will have a forward model $F_k=I$ for $k=1,2,...,g$.  Observations are noisy
pressure reading at \(p\) sensors given by
\begin{equation}
{\bf y}_k = H_k u^t + {\bf v}_k / k,\quad {\bf v}_k \sim\mathcal N(0,\sigma_y^2 I),
\end{equation}
where $H_k = I$ for $k=1,2,...,g$.  We assume that, in time, our pressure readings become less noisy.

We exponentiate $\log \kappa_0^{(j)}$ to enforce positivity of $\kappa_0^{(j)}$ and apply $F_0$ to the ensemble to solve for the state estimations \(u_1^{(j)}\).  The setup is discretized by conforming
$\mathbb{P}_1$ finite elements on a uniform triangulation of mesh size \(h\) and solved using the continuous Galerkin Finite Element Method (FEM). We make use of the FEniCSx platform for the FEM computation.  We denote the approximated pressure solution at assimilation cycle $k$ as ${\bf u}_{h,k}$ and the discretized true solution as ${\bf u}_{h}^t$. Moreover, linear
systems are solved by conjugate gradients with AMG preconditioning.  Then we carry out CGD-EnKF assimilating ${\bf y}_k$ each cycle for $k=1,2,...,g$ {omitting further preconditioning as cond($\Psi_z$) is reasonably small under the conditions described below}.

For our experiment, let $h=0.125$ such that $m=81$.  We set $p=81$ to observe the noisy pressure readings at all discretized points.  Let $\sigma_\kappa^2=0.1$ with $L=0.3$, $\sigma_y^2=0.03$, and $g=10000$ CGD-EnKF iterations.  Additionally, let $\lambda=1.02$ be the multiplicative covariance inflation in the forecast step.  We will also carry out the same experiment described using classical EnKF to verify the results.

Performance is reported via pointwise errors in \(u\), the relative \(L^2\)-errors in \(u\) and the normalized data misfit \(\|H_k\bar{\bf u}_{h,k}-{\bf y}_k\|_{R^{-1}}\).  This
benchmark provides a reproducible setting focused on pressure
identification under observational uncertainty.  Note that since $\log \kappa_0^{(j)}$ is not updated in this first setup, we do not track the permeability.

Figure \ref{fig:3_1_perf} demonstrates that (a) the pressure field solution from CGD-EnKF closely approximates the true pressure field (b) when given an incorrect initial pressure field (c).  The incorrect pressure field is obtained by solving the Darcy problem given an initial ensemble for permeability that has a mean initial permeability field (d) and a true permeability field (e). Figure \ref{fig:CGD-EnKFvsEnKF} further supports the accuracy by showing a reduction of $L^2$ error and data misfit for the CGD-EnKF pressure field solution as $k$ increases.  Additionally, EnKF provides similar error for these quantities over all $k$ assimilation cycles giving further validation of our method.  Therefore, we demonstrate in this example that our CGD-EnKF method can achieve a high degree of accuracy for problems with practical applications in fields such as porous media.

\begin{figure}[htbp]
    \centering
    \begin{subfigure}[t]{0.35\textwidth}
        \centering
        \includegraphics[width=\linewidth]{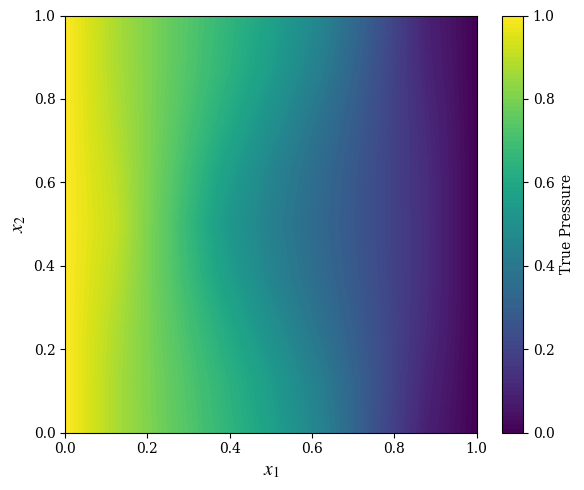}
        \caption{${\bf u}_h^t$}
    \end{subfigure}
    \begin{subfigure}[t]{0.35\textwidth}
        \centering
        \includegraphics[width=\linewidth]{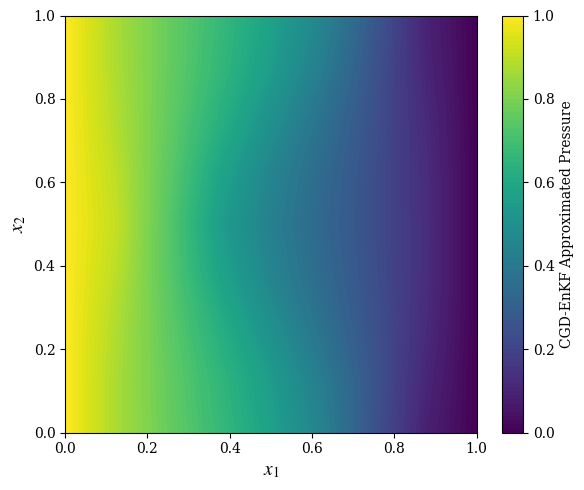}
        \caption{$\bar{{\bf u}}_{h,10000}$ from CGD-EnKF}
    \end{subfigure}
    \begin{subfigure}[t]{0.35\textwidth}
        \centering
        \includegraphics[width=\linewidth]{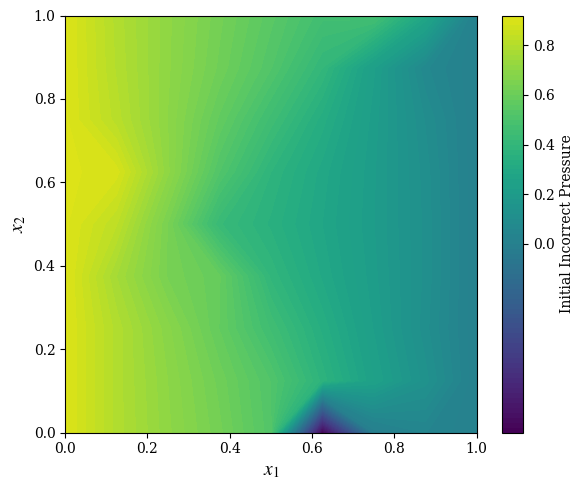}
        \caption{$\bar{{\bf u}}_{h,0}$}
    \end{subfigure}
    \vfill
    \begin{subfigure}[t]{0.35\textwidth}
        \centering
        \includegraphics[width=\linewidth]{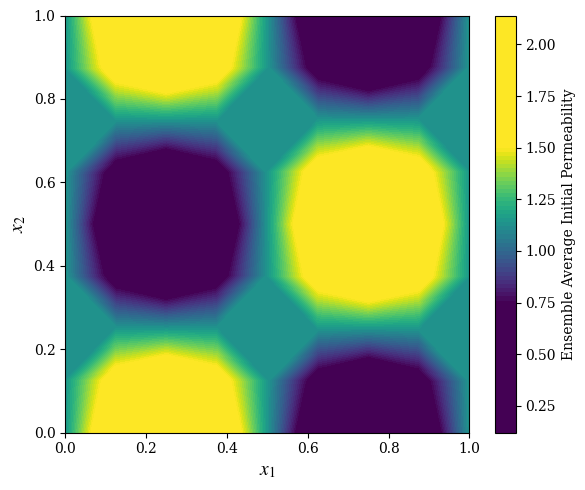}
        \caption{$\kappa^t$}
    \end{subfigure}
    \begin{subfigure}[t]{0.35\textwidth}
        \centering
        \includegraphics[width=\linewidth]{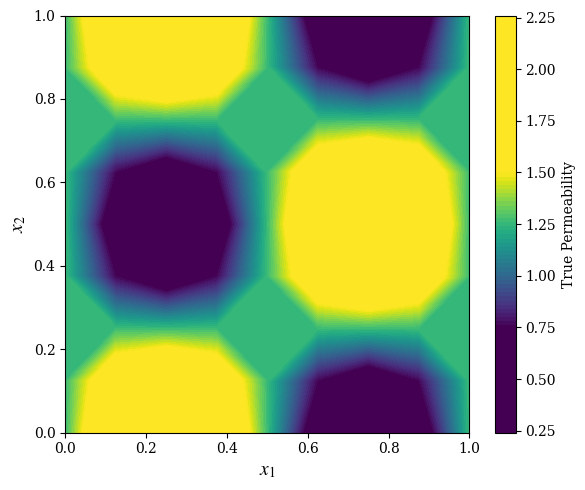}
        \caption{$\bar{\kappa}_0$}
    \end{subfigure}
    \caption{Example 3. Domain Visualization}
    \label{fig:3_1_perf}
\end{figure}

\begin{figure}[htbp]
\centering
    \begin{subfigure}[t]{0.44\textwidth}
        \centering
        \includegraphics[width=\linewidth]{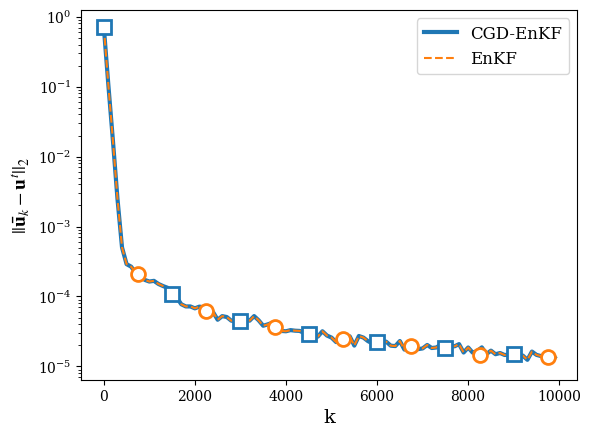}
        \caption{$|| \vect{\bar{u}}_{h,k} - {\bf u}_h^t ||_2$ over $10000$ DA cycles}
    \end{subfigure}
    \begin{subfigure}[t]{0.44\textwidth}
        \centering
        \includegraphics[width=\linewidth]{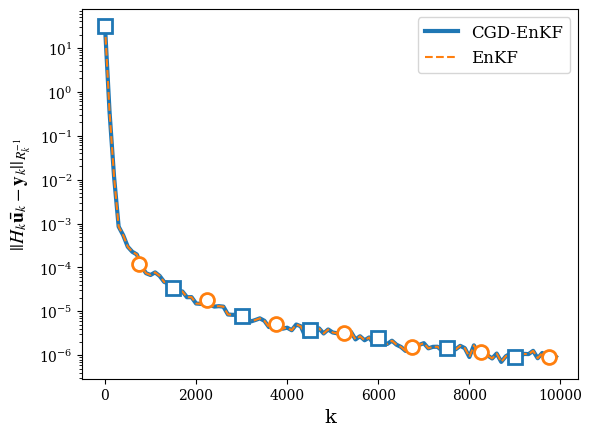}
        \caption{Data Misfit over $10000$ DA cycles}
    \end{subfigure}
    \caption{Example 3. Performance comparing CGD-EnKF and EnKF}
    \label{fig:CGD-EnKFvsEnKF}
\end{figure}

\section{Conclusions}
\label{sec:conclusion}

In this work, we designed a reformulation of EnKF with CGD to create the parallelizable CGD-EnKF algorithm.  We demonstrated that CGD-EnKF is computationally as efficient as sEnKF for high dimensionality.  Additionally, we show that in the presence of error correlations between observation components, reflected in $R$, the CGD-EnKF may outperform serial EnKF and perform near EnKF (up to the CG error) in terms of accuracy.  We provide the error analysis for the CGD algorithm and propose a sharper CGD bound that is proved and validated.  We also demonstrate the further improved computational efficiency on highly dimensional systems of the algorithm CGD-EnKF-Reduced under the specific case where $N \ll p$. Finally, we demonstrated the applicability of CGD-EnKF to several examples, including complicated nonlinear PDE problems such as Darcy Flow by validating that the accuracy of CGD-EnKF nearly coincides with the EnKF state estimation.

This work motivates future directions in applying our reformulation to the Ensemble Kalman Inversion(EKI) method or its derivative methods and solving high dimensional inverse problems in a more computational efficient manner.  The proposed analytical bounds for CGD also suggest future exploration.

\section*{Acknowledgments}
The work of Lee and Valyou was partially supported by the U.S. Department of Energy, Office of Science, Energy Earthshots Initiatives under Award Number DESC-0024703 and by the U.S. National Science Foundation under Grant DMS-2208402.  The work of Liu was supported in part by the U.S. National Science Foundation (DMS-2208249); The work of Zikatanov was supported in part by the U.S. National Science Foundation (DMS-2208249) and by the Center for Advanced Studies at the Norwegian Academy of Science and Letters, Norway under the program Mathematical Challenges in Brain Mechanics (2025/2026).

\section*{Appendix A. Derivation of the sEnKF from EnKF}
\label{sec:app}

Below the Algorithm~\ref{enkf-serial-0} is the sEnKF analysis step written in the equivalent matrix notation.
 \begin{algorithm}[H]
\caption{Serial EnKF Analysis (Matrix Formulation)\label{enkf-serial-0}}
\begin{algorithmic}[1]
\Function{EnKF\_serial}{$X^f$}
\State     Set \(Z = H X^f\);
\State $X^a= X^f + X^f\widehat{Z}^T
\left[\operatorname{diag}\left(\widehat{Z}\widehat{Z}^T + R_{N}\right)\right]^{-1}
(Y - Z)$
\State \Return{$X^a$}
\EndFunction
\end{algorithmic}
\end{algorithm}

Below, we derive sEnKF from EnKF. We first derive some simple identities and then show that sEnKF (see Algorithm~\ref{enkf-serial-0})is obtained by approximating the inverse involved in the exact EnKF (see Algorithm~\ref{enkf-0})  by a diagonal matrix. 
We first provide several identities  which show that the matrix formulation used for Algorithm~\ref{enkf-serial-0}  is equivalent to the form shown in Figure \ref{fig:comparison} (right).

If $\bm{z}_j$ denotes the $j$-th row of $Z$ (written as a column vector with $N$-elements) we have
\begin{equation}\label{eq:components}
\begin{aligned}
&\bm{z}_j=Z^T\hat{\vect{e}}_j,\quad
z_{(j,i)}=\hat{\vect{e}}_i^TZ^T\hat{\vect{e}}_j\quad\mbox{and}\quad 
\overline{z}_i=
\frac{1}{N}\sum_{l=1}^N z_{(j,l)} =
\left[Q_{\mathbbm{1}}\bm z_j\right]_l=
\hat{\vect{e}}_l^TQ_{\mathbbm{1}}\bm z_j\in \mathbb{R},\\
&(z_{(j,i)} - \bar{z}_{i})=\hat{\vect{e}}_i^T(\bm z_{j} - Q_{\mathbbm{1}}\bm z_{j})
=\hat{\vect{e}}_i^T(I - Q_{\mathbbm{1}})
\bm z_{j})=\hat{\vect{e}}_i^T(I - Q_{\mathbbm{1}})Z^T\hat{\vect{e}}_{j}\in \mathbb{R}.
\end{aligned}
\end{equation}

Next, using these identities and 
the fact that 
\(\bm{x}_{i} - \overline{\bm x}=X(I-Q_{\mathbbm{1}})\bm{e}_i\) we arrive at 
\begin{equation*}
\begin{aligned}
\sum_{i=1}^N
(\bm{x}_{i} - \overline{\bm x})(z_{(j,i)} - \bar{z}_{i})
&=
\sum_{i=1}^NX(I-Q_{\mathbbm{1}})\hat{\vect{e}}_i\hat{\vect{e}}_i^T(\bm z_{j} - Q_{\mathbbm{1}}\bm z_{j})
=X(I-Q_{\mathbbm{1}})\bm z_{j}.
\end{aligned}
\end{equation*}

Finally, to show that the serial formulations are equivalent, we need an identity for the diagonal elements of \((\widehat{Z}\widehat{Z}^T+R_N)\).

Indeed, if we employ the relations from~\eqref{eq:components} we obtain the relation that is needed:
 \begin{equation}\label{indentity-diag}
 \begin{aligned}
\left[ \operatorname{diag}
\left(\widehat{Z}\widehat{Z}^T+R_N\right)\right]_{jj}
&=
R_{N,jj}+\left(\widehat{Z}^T\hat{\vect{e}}_j\right)^T\left(\widehat{Z}^T\hat{\vect{e}}_j\right)\\
&=
R_{N,jj}+\left((I-Q_{\mathbbm{1}})Z^T\hat{\vect{e}}_j\right)^T\left((I-Q_{\mathbbm{1}})Z^T\hat{\vect{e}}_j\right)\\
&=
R_{N,jj}+\sum_{i=1}^N\left(\hat{\vect{e}}_i^T(I-Q_{\mathbbm{1}})Z^T\hat{\vect{e}}_j\right)^2\\
&=R_{N,{jj}}+\sum_{i=1}^{N} (z_{(j,i)} - \bar{z}_{{i}})^2
\end{aligned}     
\end{equation}

Furthermore, from the exact EnKF, we have that
\begin{equation}\label{X-Z}
\begin{aligned}
X^a&=
X^f+X^f(I-Q_{\mathbbm{1}})Z^T\left(
R_N+Z(I-Q_{\mathbbm{1}})Z^T\right)^{-1}\\
&=X^f + X^f\widehat{Z}^T\left(\widehat{Z}\widehat{Z}^T+R_N\right)^{-1}(Y-Z)
\end{aligned}
\end{equation}

Thus, the sEnKF approximation is:
\begin{equation}
\begin{aligned}
X^a &= X^f + X^f\widehat{Z}^T\left[\operatorname{diag}
\left(\widehat{Z}\widehat{Z}^T+R_N\right)
\right]^{-1}(Y - Z)
\end{aligned}
\end{equation}

\bibliographystyle{elsarticle-num} 
\bibliography{reference}

\end{document}